\documentclass[a4paper]{extarticle}
\usepackage[utf8]{inputenc}
\usepackage{color}
\usepackage{enumerate}
\usepackage{comment}
\usepackage{hyperref}
\usepackage{authblk}
\usepackage{cancel}
\RequirePackage{amsthm,amsmath,amsfonts,amssymb}
\usepackage[normalem]{ulem}
\newtheorem{Theorem}{Theorem}[section]
\newtheorem{Definition}[Theorem]{Definition}
\newtheorem{Proposition}[Theorem]{Proposition}
\newtheorem{Assumption}{Assumption}
\newtheorem{Lemma}[Theorem]{Lemma}

\newtheorem{Corollary}[Theorem]{Corollary}
\newtheorem{Remark}[Theorem]{Remark}
\newtheorem{Example}[Theorem]{Example}

\def \N{\mathbb{N}}
\def \R{\mathbb{R}}

\def \E{\mathbb{E}}

\def \P{\mathbb{P}}

\def \D{\mathbb{D}}

\def \Fc{{\cal F}}

\def \Hc{{\cal H}}

\def \Sc{{\cal S}}
\def \Tc{{\cal T}}

\def\beqs{\begin{eqnarray*}}
\def\enqs{\end{eqnarray*}}
\def\beq{\begin{eqnarray}}
\def\enq{\end{eqnarray}}

\newcommand{\nc}{\newcommand}
\nc{\esssup}{\mathop{\mathrm{ess\;sup}}}

\newcommand{\1}[1]{{\bf 1}_{\{#1\}}}

\title{Malliavin calculcus for a Hawkes process}
\date{\today}

\author[1]{Dorian Cacitti-Holland \thanks{Dorian.Cacitti\_Holland.Etu@univ-lemans.fr}}
\author[1]{Laurent Denis \thanks{Laurent.Denis@univ-lemans.fr}}
\author[1]{Alexandre Popier\thanks{Alexandre.Popier@univ-lemans.fr}}
\affil[1]{\small Laboratoire Manceau de Math\'ematiques, Le Mans Universit\'e, Avenue O. Messiaen, 72085 Le Mans cedex 9, France.  }

\begin{document}

\maketitle

\begin{abstract}
We develop a Malliavin calculus for nonlinear Hawkes processes in the sense of Carlen and Pardoux. This approach, based on perturbations of the jump times of the process, enables the construction of a local Dirichlet form. As an application, we establish criteria for the absolute continuity of solutions to stochastic differential equations driven by Hawkes processes. We also derive sensitivity formulas for the valuation of financial derivatives with respect to model parameters.
\end{abstract}



\renewcommand*\contentsname{Summary}

\tableofcontents

\section{Introduction}

In this paper we aim to develop a local Malliavin calculus with respect to a Hawkes process. Malliavin calculus is a mathematical framework used to study the smoothness of random variables and functionals of stochastic processes, especially those driven by Brownian motion. The central concept is the Malliavin derivative, a type of derivative that extends the classical notion of differentiation to the space of random variables. This approach allows one to analyze the regularity and differentiability of random processes, providing powerful tools for studying stochastic differential equations (SDEs) and probabilistic systems. The key ideas of Malliavin calculus were introduced by Paul Malliavin in the 1970s. There is a huge literature on this subject, encompassing Malliavin calculus for L\'evy processes (see among others \cite{nual:06,bich:jaco:87,dinunno:09,bouleau:denis} and the references therein). 

Hawkes processes have been introduced by Alan Hawkes in the 1970s as a class of self-exciting point processes. They are widely used to model events or occurrences where the occurrence of one event increases the likelihood of subsequent events in the near future. These processes have been used to model earthquakes, and for some time now, they have been experiencing a renewed interest due to their applications in finance and actuarial science \cite{bacry:15,daley:vere:jones:2,laub:lee:poll:25,laub:lee:taimre}. 

A Hawkes process is characterized by an intensity function that depends on both a baseline intensity and a history of past events. More specifically, the intensity at time $t$ is given by:
$$\lambda^*(t) = \lambda_t + \gamma \left( \int_{(0,t)} \mu(t-s) dN_s\right)$$
where $\lambda_t$ is the baseline intensity, $\gamma$ and $\mu$ are functions that describe the impact of past events, and $N$ is the counting process that represents the occurrences of events. A special case of interest is the linear case where $\gamma (x)=x$. The key feature of Hawkes processes is that the function $\mu$ is typically a non-negative function, which means that past events increase the probability of future events — hence the term ``self-exciting". 

\bigskip

Recently a Malliavin calculus with respect to the Hawkes process $N$ is developed in \cite{hillairet:huang:khabou:reveillac} and \cite{hillairet:reveillac:rosenbaum}. Roughly speaking, the main ingredient consists to perturb the system by adding a particle (or a jump) (see  \cite[Lemma 3.5]{hillairet:reveillac:rosenbaum}) leading to an expansion formula for functionals of the Hawkes process (\cite[Theorem 3.13]{hillairet:reveillac:rosenbaum}). Let us mention that the derivative operator is not local. With these results, they are able to develop a Stein method (see \cite{hillairet:huang:khabou:reveillac}) and to compute some prices of financial or insurance derivatives (see \cite[Section 4]{hillairet:reveillac:rosenbaum}).

\bigskip

Our method is different and follows the approach of Carlen-Pardoux \cite{Carlen1990}. We perturb the jump times and formally differentiate with respect to these jump times. This allows us to define a local derivative, satisfying the chain rule. We apply it to the study of absolute continuity for the law of Hawkes functionals and to the computation of Greeks. 

\paragraph{Breakdown.}

The first step of our construction of a Malliavin derivative with respect to a Hawkes process is to define a directional derivative with respect to a function $m \in \mathcal{H}$, where $\mathcal{H}$ is the ad hoc Cameron-Martin space. 
An integration by parts formula is obtained thanks to the absolute continuity property of the law of the perturbed jump times w.r.t. the initial probability measure (see  Proposition \ref{functCalculus}, Theorem \ref{th:DmF} and Proposition \ref{propo:closable}).

The second step is to define the Malliavin derivative $D$ in all directions by considering a Hilbert basis of $\mathcal H$. We obtain 
a local Dirichlet form $(\mathbb{D}^{1,2}, \mathcal{E})$ which admits a carré du champ $\Gamma$ 
and a gradient $D$ (see Proposition \ref{propo:local:dirichlet}). 
Therefore we get similar properties to the directional derivative as the chain rule. 
Moreover we are interested in the associated divergence operator $\delta$, for which we get an explicit expression of $\delta(u)$ when $u$ is predictable
(see Proposition \ref{FormuleDiv}, Remark \ref{remark:resume} and Corollary \ref{coro:predictable}).

We then establish an absolute continuity criterion: conditionally to $\Gamma[F] = (\Gamma[F_i,F_j])_{1\leq i,j\leq d} \in GL_d(\mathbb{R})$, the random vector $F = (F_1, \cdots, F_d) \in (\mathbb{D}^{1,2})^d$ admits a absolutely continuous law with respect to the Lebesgue measure on $\mathbb{R}^d$ (see Theorem \ref{thm:abs_cont} and Corollary \ref{coro:criter:density}). 

This criterion is firstly applied to the solution of stochastic differential equation driven by the Hawkes process (see Theorem \ref{thm:density:d=1}, Corollary \ref{coro:wronskian} and  Proposition \ref{propo:spanning}). 
As a second application, we compute Greeks for a financial payoff when the underlying process is driven by the Hawkes process (see Proposition \ref{propo:delta}). 


\bigskip

\section{Framework and directional derivation}{\label{Section2}}

\subsection{Setting and density of the jump times} \label{ssect:setting}

We consider the probability space $(\Omega, \Fc ,\P)$ where $\Omega$ is the space of càdlàg (right continuous with left limit) trajectories
$$\omega(t)=\sum_i i \mathbf{1}_{[t_i ,t_{i+1})},$$
with $0<t_1<\cdots <t_i<\cdots $. We define 
$$N_t (\omega)= \sum_{s\leq t} \Delta \omega_s, \quad t\geq 0,$$
the process which counts the jumps between $0$ and $t$, where $\Delta w_s = w_s - w_{s-}$  and $\displaystyle w_{s-} = \lim_{u \to s, u < s} w_u$. We assume that, under $\P$, $(N_t)_{t\in \mathbb{R}_+}$ is a {non-linear} Hawkes process with conditional intensity 
$$\lambda^* (t)= \lambda_t+\gamma \left(  \int_{(0,t)} \mu(t-s) dN_s \right) = \lambda_t +  \gamma \left( \sum_{i=1}^{{N_{t-}}} \mu(t-T_i) \right).$$
Throughout this paper, we suppose that 
\begin{Assumption}\label{ass:lambda_mu}
$\ $
\begin{itemize}
 \item {$\lambda : \mathbb{R}_+ \longrightarrow \mathbb{R}_+^*$ is differentiable with bounded derivative and bounded from below by a constant  $\lambda_* \in \mathbb{R}_+^*$ : $\lambda_t \geq \lambda_*>0$ for any $t\in \mathbb{R}_+$.}
   \item $\mu : \mathbb{R}_+ \longrightarrow \mathbb{R}_+$ is of class $C^1$ with bounded derivative and belongs to  $L^1 ([0,\infty);dt)$.
   \item $\gamma : \mathbb{R}_+ \longrightarrow \mathbb{R}_+$ is of class $C^1$ with bounded derivative and  we denote 
   $$a=\sup_{x\in \R}\vert \gamma^\prime (x)\vert.$$
   \item  $\gamma (0)=0$.
    \item The next inequality holds:
    $$a \lVert \mu \rVert_1 = a \int_0^{\infty} \mu(t) dt < 1.$$
\end{itemize}
\end{Assumption}
\begin{Remark}
    Under this set of assumptions, to see that such a non-linear Hawkes process exists, we can follow and adapt in an easy way the algorithmic construction based on a Poisson random measure on $(0,\infty)\times (0,\infty)$ with unit intensity in \cite[Appendix A1]{costa:graham}. See also \cite[Section 12.6]{Bremaud-2020}. 
\end{Remark}

For any $n\in \N^*$, $0< t_1 < \cdots <t_n$ and $s \in \mathbb{R}_+^*$, we put:
\begin{equation} \label{def:lambda_etoile}
\lambda^* (s;t_1,\cdots,t_n)=\lambda_s+\gamma \left(  \sum_{i=1}^n \mu(s-t_i) \mathbf{1}_{\{s > t_i\}} \right).
\end{equation}
We introduce $(T_i)_{i\in \mathbb{N}^*}$ the  increasing sequence of jump instants of the Hawkes process $N$. Thus, for any $n\in \mathbb{N}^*$ and $t>0$, on the event $\{N_t = n\}$ for all $s\in [0,t]$
$$\lambda^*(s)=\lambda^*(s; T_1, \cdots, T_n).$$
From now on, $T \in \mathbb{R}_+^*$  is a fixed time horizon and we  consider the $\mathbb{P}$-complete right continuous filtration $(\mathcal{F}_t)_{0\leq t\leq T}$ generated by the Hawkes process $N$.

Let $n\in\N^*$ and $0<t_1<t_2<\cdots <t_n$, knowing $T_1=t_1,\ldots ,T_n=t_n$ the process $(N_t -N_{t_n}, \ t \in [t_n, T_{n+1}[)$ is an inhomogeneous Poisson process with intensity $\lambda^* (\cdot;t_1,\cdots ,t_n)$. We deduce the following conditional link for $t\geq t_n$:
\begin{eqnarray*}
    \P(T_{n+1}>t| T_1=t_1 ,\cdots, T_n=t_n) & = & \P(N_t-N_{t_n}=0| T_1=t_1 ,\cdots, T_n=t_n)
    \\ & = & e^{-\int_{t_n}^t \lambda^* (s;t_1,\cdots ,t_n)\,ds}.
\end{eqnarray*} 
Thus the density of $T_{n+1}$ knowing $T_1=t_1,\ldots ,T_n=t_n$ is
$$t\longmapsto \lambda^* (t;t_1,\cdots ,t_n)e^{-\int_{t_n}^t \lambda^* (s;t_1,\cdots ,t_n)\,ds}\mathbf{1}_{\{t>t_n\}}.$$
We deduce that 
$(T_1,\cdots ,T_n, T_{n+1})$ admits for density
\begin{align} \label{eq:density_jump_times}
(t_1 ,\cdots , t_{n+1})\longmapsto &  \left(\prod_{i=1}^{n+1}\lambda^* (t_i;t_1,\cdots ,t_n)\right) \\ \nonumber
&\qquad \times \exp\left( -\int_{0}^{t_{n+1}} \lambda^* (s;t_1,\cdots ,t_n)\,ds \right) \mathbf{1}_{\{0<t_1<\cdots <t_{n+1}\}}, 
\end{align}
where we used the fact that $\lambda^* (s;t_1,\cdots ,t_n)=\lambda^* (s;t_1,\cdots ,t_{i-1})$ if $s\leq t_i$.


\begin{Lemma} \label{lemma:density}
    The distribution of $(T_1, \cdots, T_n)$ conditionally to $\{N_T = n\}$ has a density
    \begin{equation} \label{eq:cond_dens_jump_times}
    \begin{array}{crcl}
        k_n : & \mathbb{R}^n & \longrightarrow & \mathbb{R}_+ \\
        & (t_1, \cdots, t_n) & \longmapsto & \dfrac{\kappa(t)}{\int_{\mathbb{R}^n} \kappa(s) ds}{=\dfrac{\kappa(t)}{\mathbb{P} (N_T =n)}}
    \end{array}
    \end{equation}
    with, for any $t = (t_1, \cdots, t_n) \in \mathbb{R}^n$,
    $$\kappa(t) = \mathbf{1}_{\{0<t_1<\cdots < t_n \leq T\}} \left(\prod_{i=1}^n \lambda^*(t_i; t_1, \cdots, t_n) \right) e^{- \int_0^T \lambda^*(s;t_1, \cdots, t_n) ds}.$$
\end{Lemma}

\begin{proof}
    Let $f$ be a {bounded} measurable function on $\mathbb{R}^n$. Then
    \begin{eqnarray*}
        && \mathbb{E}[f(T_1, \cdots, T_n) \mid N_T = n]
        \\ && =  \dfrac{\mathbb{E}[1_{\{N_T = n\}} f(T_1, \cdots, T_n)]}{\mathbb{P}(N_T = n)}
        \\ && = \dfrac{\mathbb{E}[1_{\{T_n \leq T < T_{n+1}\}} f(T_1, \cdots, T_n)]}{\mathbb{P}(T_n \leq T < T_{n+1})}
        \\ && = \dfrac{\int_{0< t_1 < \cdots < t_n \leq T < t_{n+1}} f(t_1, \cdots, t_n) \varphi(t_1, \cdots, t_{n+1}) dt_1 \cdots dt_{n+1}}{\int_{0<t_1 < \cdots < t_n \leq T < t_{n+1}} \varphi(t_1, \cdots, t_{n+1}) dt_1 \cdots dt_{n+1}}
    \end{eqnarray*}
    with from \eqref{eq:density_jump_times} {
    \begin{align*}
        & \varphi(t_1, \cdots, t_{n+1})
        = \left(\prod_{i=1}^{n+1} \lambda^*(t_i;t_1, \cdots, t_n) \right) e^{- \int_0^{t_{n+1}} \lambda^*(s;t_1, \cdots, t_n) ds}
        \\ & \quad =  \left(\prod_{i=1}^n \lambda^*(t_i; t_1, \cdots, t_n) \right) e^{- \int_0^{t_n} \lambda^*(s;t_1, \cdots, t_n) ds}
        \\ & \qquad \times  \lambda^*(t_{n+1}; t_1, \cdots, t_n)e^{- \int_{t_n}^{t_{n+1}} \lambda^*(s; t_1, \cdots, t_n) ds}.
    \end{align*}
    Thus
    \begin{align*}
        & \int_{T< t_{n+1}} \varphi(t_1, \cdots, t_{n+1}) dt_{n+1}
        \\& \quad = { \left(\prod_{i=1}^n \lambda^*(t_i; t_1, \cdots, t_n) \right) e^{- \int_0^{t_n} \lambda^*(s;t_1, \cdots, t_n) ds}}\int_T^{\infty} \lambda^*(t_{n+1}; t_1, \cdots, t_n) e^{- \int_{t_n}^{t_{n+1}} \lambda^*(s; t_1, \cdots, t_n) ds} dt_{n+1}
        \\ &\quad =  { \left(\prod_{i=1}^n \lambda^*(t_i; t_1, \cdots, t_n) \right) e^{- \int_0^{t_n} \lambda^*(s;t_1, \cdots, t_n) ds}}e^{-\int_{t_n}^T \lambda^*(s; t_1, \cdots, t_n) ds}
        \\ &\quad  = { \left(\prod_{i=1}^n \lambda^*(t_i; t_1, \cdots, t_n) \right) e^{- \int_0^{T} \lambda^*(s;t_1, \cdots, t_n) ds}}.
    \end{align*}
    Assumption \ref{ass:lambda_mu} ensures that all integrals are finite. Therefore
$$
        \mathbb{E}[f(T_1, \cdots, T_n) \mid N_T = n] = \dfrac{\int_{\mathbb{R}^n} f(t) \kappa(t) dt}{\int_{\mathbb{R}^n} \kappa(t) dt}
      = \int_{\mathbb{R}^n} f(t) k_n(t) dt
  $$
  with, for any $t = (t_1, \cdots, t_n) \in \mathbb{R}^n$,
    \begin{eqnarray*}
         \kappa(t)
         & = &  \mathbf{1}_{\{0<t_1<\cdots < t_n \leq T\}} \left(\prod_{i=1}^n \lambda^*(t_i; t_1, \cdots, t_n) \right) e^{- \int_0^T \lambda^*(s;t_1, \cdots, t_n) ds},
    \end{eqnarray*}
which implies the conclusion of the lemma. }
\end{proof}

Note that for any $0<t_1<\cdots < t_k < T$ and any $j \in \{1, \cdots, k\}$,
        {\begin{eqnarray*}
            \lambda^* (t_j;t_1,\cdots ,t_{k-1}) & = & \lambda_t + \gamma \left(\sum_{i=1}^{j-1} \mu(t_j -t_i)\right) \leq \lambda^T + n a\lVert \mu \rVert_{\infty},
        \end{eqnarray*}
        where $\lambda^T =\sup_{t\in [0,T]}\lambda_t$.} Thus the density $k_n$ of $(T_1, \cdots, T_n)$ knowing $\{N_T = n\}$ is bounded by
\begin{equation} \label{eq:bound_dens_jump_times}
            (t_1 ,\cdots , t_{k})\longmapsto \dfrac{1}{\mathbb{P}(N_T = n)}\left(\lambda^T + na \lVert \mu \rVert_{\infty}\right)^n \mathbf{1}_{\{0<t_1<\cdots <t_{n} < T\}}.
        \end{equation}

\subsection{Cameron-Martin space and directional derivation}

We apply the same approach as in \cite{Carlen1990} to define the directional derivative using the reparametrization of time with respect to a function in a Cameron-Martin space.

Let $L^2([0,T])$ be the usual space of square integrable functions on $[0,T]$ with respect to the Lebesgue measure and $\mathcal{H}$, {  which will play the role of {\it Cameron-Martin space}, }be the closed subspace of $L^2([0,T])$ orthogonal to the constant functions, i.e.,
\begin{equation}\label{eq:def_mathcal_H}
    \mathcal{H} = \left\{m \in L^2([0,T]) \quad \int_0^T \, {m}(s)ds = 0\right\}. 
\end{equation}
We denote $\widehat{m} = \int_0^. \, {m}(s)ds$ for every $m\in \Hc$, then $\widehat{m}(0) = \widehat{m}(T) =0$. In a natural way, $\mathcal{H}$ inherits the Hilbert structure of $ L^2([0,T])$ and we denote by $\| \, \|_{\mathcal{H}}$ and $\langle\cdot ,\cdot  \rangle_{\mathcal{H}}$ the norm and the scalar product on it. From now on in this section, we fix a function ${m} \in \Hc$. The condition $\int_0^T \, {m}(s)ds = 0$  ensures that the change of intensity that we are about to define simply shifts the jump times without affecting the total number of jumps. {In order to avoid jumps crossing, we need to truncate $m$, so we }  define   {for any $\varepsilon>0$}
\[ \widetilde{m}_{\varepsilon}(s) = \left\{ \begin{array}{ll}
         -\frac{1}{3\varepsilon} & \mbox{if ${m}(s) \leq -\frac{1}{3\varepsilon}$} ;\\
        m(s) & \mbox{if $ \frac{-1}{3\varepsilon} \leq {m}(s) \leq \frac{1}{3\varepsilon}$} ;\\
         \frac{1}{3\varepsilon} & \mbox{if $ {m}(s) \geq \frac{1}{3\varepsilon}$} . \end{array} \right. \]
and $m_{\varepsilon}\in \Hc$ such that 
\begin{equation}\label{eq:def_m_eps}
{m}_{\varepsilon}(s) = \widetilde{m}_{\varepsilon}(s) -\frac{1}{T}
\int_0^T\,\widetilde{m}_{\varepsilon}(s) ds. 
\end{equation}
We have the following inequality: $\frac{1}{3}\leq 1+ \varepsilon
{m}_{\varepsilon}(s) \leq \frac{5}{3}$ (since $- \frac{1}{3\varepsilon}\leq
\widetilde{m}_{\varepsilon}(s) \leq \frac{1}{3\varepsilon}$). \\
{We clearly have:
\begin{Lemma} \label{lemma:convergence:meps}
    $\lim_{\varepsilon \rightarrow 0}\parallel {m} - m_{\varepsilon}\parallel_{\Hc} = 0.$
\end{Lemma}}
  {Let us remark that if $m$ is bounded, this truncation is not necessary since in that case, ${m}_{\varepsilon}=m$ for $\varepsilon $ small enough.}

For all $\varepsilon >0$, we define the reparametrization of time with respect
to $m_{\varepsilon}$ as follows 
\begin{equation*}
    \tau_{\varepsilon} (s) = s + \varepsilon \widehat{m}_{\varepsilon}(s) = \int_0^s (1+\varepsilon m_\varepsilon (u)) du, \quad s\in \R_+.
\end{equation*}
Notice that $\tau_{\varepsilon}(0) = 0, \tau_{\varepsilon}(T) =T$, and since $1+
\varepsilon {m}_{\varepsilon}(s) \in \left[ \frac{1}{3}, \frac{5}{3} \right] \subset \mathbb{R}_+^*$, $\tau_\varepsilon$ is an increasing function hence invertible so the number and the order of jump times between $0$ and $T$ remain unchanged. Moreover, a direct calculation gives
$$\forall s\in [0,T], \quad \tau_\varepsilon^{-1}(s)=\int_0^s \frac{1}{1+\varepsilon m_\varepsilon (\tau_\varepsilon^{-1}(u))} du.$$
Let $\mathcal{T}_{\varepsilon} : \Omega
\rightarrow \Omega$ be the map defined by: $$\forall \omega\in \Omega,(\Tc_{\varepsilon}(\omega))(s) =
\omega(\tau_{\varepsilon}(s)).$$
For all $F \in L^2(\Omega)$, we denote 
$\mathcal{T}_{\varepsilon}F = F\circ \mathcal{T}_{\varepsilon}$ and  $\P^{\varepsilon}$  the probability measure
$\P\Tc_{\varepsilon}^{-1}$ defined on $\mathcal{F}_T$.

\begin{Definition}  \label{def:D^0_m}
We denote 
$$ \mathbb{D}^0_m = \left\{F \in L^2(\Omega): \dfrac{\partial \mathcal{T}_{\varepsilon}F}{\partial \varepsilon} |_{\varepsilon = 0} = \lim_{\varepsilon \rightarrow
0} \frac{1}{\varepsilon} (\mathcal{T}_{\varepsilon}F - F) \text{ in }
L^2(\Omega) \text{ exists}\right\}.$$
For $F \in \D ^0_m, D_mF$ is defined as
the limit
\begin{equation} \label{eq:def_D_m}
D_m F = \dfrac{\partial \mathcal{T}_{\varepsilon}F}{\partial \varepsilon} |_{\varepsilon = 0} = \lim_{\varepsilon \rightarrow
0} \frac{1}{\varepsilon} (\mathcal{T}_{\varepsilon}F - F).
\end{equation}
\end{Definition}
Let define the set $\mathcal{S}$ of ``smooth" functions. 
\begin{Definition}
We say that a map $F: \Omega \rightarrow \mathbb{R}$ belongs to $\mathcal{S}$ if there exist $a \in \mathbb{R}$ and $d \in \mathbb{N}^*$ and for any $n \in \{1, \cdots, d\}$, a function $f_n : \mathbb{R}^{n} \rightarrow \mathbb{R}$ such that:
\begin{enumerate}
    \item The random variable $F$ can be written
    \begin{equation} \label{eq:smooth}
        F = a 1_{\{N_T = 0\}} + \sum_{n = 1}^d  f_{n}(T_1, T_2, \cdots, T_n) \mathbf{1}_{\{N_T = n\}}.
    \end{equation}
    \item For any $n \in \{1, \cdots, d\}$, the function $f_n$ is $C^{\infty}$.
\end{enumerate}
\end{Definition}
\begin{Remark} \label{remark:density}
    The space $\mathcal{S}$ is dense in $L^2(\Omega, \mathcal{F}_T,\P)$.
\end{Remark}

Here are some basic properties of directional derivatives on $\Sc$.

\begin{Lemma} \label{lemma:DTj}
    Let $j\in \mathbb{N}^*$ and $\overline{T}_j = T_j \wedge T$. Then $\overline{T}_j \in \mathbb{D}_m^0$ and
    $$D_m \overline{T}_j = - \widehat m{ (\overline{T}_j)}=-\int_0^{\overline{T}_j}m(s)ds.$$
\end{Lemma}

\begin{proof}
    We first remark that for any $\omega \in \Omega$, if  $T_j (\omega)<T$
  $$ \overline{T_j} (\omega \circ \tau_\varepsilon)=T_j(\omega \circ \tau_\varepsilon)  =   \tau_\varepsilon^{-1}(T_j(\omega))=\tau_\varepsilon^{-1}(\overline{T}_j(\omega)),$$
    and if $T_j (\omega )\geq T,  \Tc_{\varepsilon}\overline{T}_j(\omega)=T$. 
    This yields
    \begin{eqnarray*}
        && |\Tc_\varepsilon \overline{T}_j(\omega) - \overline{T_j}(\omega) + \varepsilon \widehat{m}(\overline{T_j})(\omega)| = \left|(\overline{T}_j \circ \Tc_\varepsilon)(\omega) - \overline{T}_j(\omega) + \varepsilon \int_0^{\overline{T_j}(\omega)} m(t)dt \right|
        \\ && \quad = \left| \overline{T}_j(\omega \circ \tau_\varepsilon) - \overline{T}_j(\omega) + \varepsilon \int_0^{\overline{T_j}(\omega)} m(t)dt \right|
        = \left| \tau_\varepsilon^{-1}(\overline{T}_j(\omega)) - \overline{T}_j(\omega) + \varepsilon \int_0^{\overline{T_j}(\omega)} m(t)dt \right|
        \\ &&\quad = \left| \tau_\varepsilon^{-1}(\overline{T}_j(\omega)) - \tau_\varepsilon (\tau_\varepsilon^{-1}(\overline{T}_j(\omega))) + \varepsilon \int_0^{\overline{T_j}(\omega)} m(t)dt \right| \\
        && \quad = \left| \tau_\varepsilon(s_\varepsilon ) - s_\varepsilon - \varepsilon \int_0^{\tau_\varepsilon(s_\varepsilon )} m(t) dt \right|, \quad \text{with } s_\varepsilon = \tau_\varepsilon^{-1}(\overline{T}_j(\omega))
        \\ && \quad \leq \left| \tau_\varepsilon(s_\varepsilon ) - s_\varepsilon - \varepsilon \int_0^{\tau_\varepsilon(s_\varepsilon )} m_\varepsilon(t) dt \right| + \varepsilon \int_0^{\tau_\varepsilon(s_\varepsilon )} |m_\varepsilon(t) - m(t)| dt
        \\ && \quad \leq  \varepsilon \left| \int_{s_\varepsilon}^{\tau_\varepsilon(s_\varepsilon )} m_\varepsilon(t) dt \right| + \varepsilon \int_0^{\tau_\varepsilon(s_\varepsilon )} |m_\varepsilon(t) - m(t)| dt
        \\ && \quad \leq  \varepsilon {\sqrt{\tau_\varepsilon(s_\varepsilon ) - s_\varepsilon}}{\sqrt{\int_0^T |m_\varepsilon(t)|^2 dt}} + \varepsilon {\int_0^T |m_\varepsilon(t) - m(t)| dt}.
    \end{eqnarray*}
    We have 
    \begin{eqnarray*}
        |\tau_\varepsilon(s_\varepsilon ) - s_\varepsilon|&=&\overline{T}_j (\omega )-s_\varepsilon
        = \int_0^{\overline{T}_j (\omega)} \left(1-\frac{1}{1+\varepsilon m_\varepsilon (\tau_\varepsilon^{-1}(u))}\right)du
        \\ & \leq & \int_0^T \left(1-\frac{1}{1+\varepsilon m_\varepsilon (\tau_\varepsilon^{-1}(u))} \right) du
         \underset{\varepsilon \to 0}{\longrightarrow} 0.
    \end{eqnarray*}
    Moreover $\lim_{\varepsilon \rightarrow 0}\int_0^T |m_\varepsilon(t) - m(t)| dt=0$ and $\int_0^T |m_\varepsilon(t)|^2 dt$ is bounded so that we get by a dominated convergence argument that $\overline{T}_j$ belongs to $\mathbb{D}_m^0$ and $D_m \overline{T}_j=-\widehat{m}(\overline{T}_j)$.
\end{proof}
 The next two propositions  are standard, so that we omit the proof.
\begin{Proposition} \label{prop:lemma}
    Let $n \in \mathbb{N}^*$ and  $f: \R^n\rightarrow \R$ a function of class $C^1$. Then $f(\overline{T}_1, \overline{T}_2, \cdots, \overline{T}_n)$ belongs to $\mathbb{D}^0_m$ and
    $$ D_m f(\overline{T}_1, \overline{T}_2, \cdots, \overline{T}_n) = - \sum_{j=1}^n \dfrac{\partial f}{\partial t_j}(\overline{T}_1, \overline{T}_2, \cdots, \overline{T}_n)\, \widehat{m}(\overline{T}_j).$$
    Thus $\Sc \subset \mathbb{D}^0_m$ and for any $F \in \mathcal{S}$ of the form \eqref{eq:smooth},
    $$D_m F = - \sum_{n=1}^d \sum_{j=1}^n \dfrac{\partial f_n}{\partial t_j}(T_1, \cdots, T_n) \widehat{m}(T_j) \mathbf{1}_{\{N_T = n\}}.$$
\end{Proposition}
\begin{Proposition}\label{functCalculus}~
    \begin{enumerate}
        \item If $F, G \in \Sc$ then $FG \in \Sc$ and $D_m(FG) = (D_mF)G + F(D_mG)$.
        \item Chain rule: If $F_1, F_2,\ldots, F_n \in  \Sc$ and $\Phi: \R ^n \rightarrow \R$ is { {a smooth function} } then $$\Phi(F_1, F_2,\cdots, F_n) \in  \Sc$$ and
        $$ D_m \Phi(F_1, F_2,\cdots, F_n) = \sum_{j=1}^n \dfrac{\partial \Phi}{\partial x_j}(F_1, F_2,\cdots, F_n) D_m F_j.$$
    \end{enumerate}
\end{Proposition}

\subsection{Absolute continuity of $\mathbb{P}^\varepsilon$ w.r.t. $\mathbb{P}$}

Let $n \in \N^*$ and $\varepsilon>0$. Evoke that  $\P^{\varepsilon}=\P\Tc_{\varepsilon}^{-1}$ and let $\E^{\varepsilon}$ be the expectation under it.  We define
\begin{equation} \label{eq:Phieps}
    \Phi_{\varepsilon}(u_1,\cdots, u_n) = (u_1 + \varepsilon 
\widehat{m}_{\varepsilon}(u_1), \cdots, u_n + \varepsilon \widehat{m}_{\varepsilon}(u_n)),
\end{equation}
\begin{equation} \label{eq:phin}
    \varphi_n (t_1,\cdots, t_n) = \bigg( \prod_{i=1}^{n}\lambda^* (t_i;t_1,\cdots ,t_n) \bigg) e^{-\int_{0}^{T} \lambda^* (s;t_1,\cdots ,t_n)ds},
\end{equation}
and 
  
\begin{equation} \label{eq:zneps}
    Z_n^{\varepsilon} = \dfrac{(\varphi_n \circ \Phi_{\varepsilon})(T_1,\cdots,T_n)}{ \varphi_n (T_1,\cdots,T_n)}\prod_{i=1}^n
(1+\varepsilon  m _{\varepsilon}(T_i)).
\end{equation}
\begin{Proposition} \label{propo:abs:cont}
    $\P^{\varepsilon}$ is absolutely continuous with respect to $\P$ with density
    $$ \dfrac{d\P^{\varepsilon}}{d\P} = \sum_{n=0}^{+\infty} Z_n^{\varepsilon}\mathbf{1}_{\{N_T =n\}}:=G^\varepsilon.$$
\end{Proposition}
\begin{proof}
    For any smooth function $f: \R
^n \rightarrow \R$, we have thanks to Lemma \ref{lemma:density}:
\begin{align*}
    & \E^{\varepsilon}[f(T_1,\cdots,T_n)\mathbf{1}_{\{N_T = n\}}] =\E[(f\circ \Phi_{\varepsilon} ^{-1})(T_1,\cdots,T_n)\mathbf{1}_{\{T_n\leq T < T_{n+1}\}}]
    \\ &\quad = \iint_{0<t_1<\cdots<t_n\leq T<t_{n+1}}\, (f\circ \Phi_{\varepsilon} ^{-1})(t_1,\cdots,t_n) \left(\prod_{i=1}^{n+1}\lambda^* (t_i;t_1,\cdots ,t_n)\right)
    \\ & \qquad \times e^{-\int_{0}^{t_{n+1}} \lambda^* (s;t_1,\cdots ,t_n)\,ds}dt_1\cdots dt_{n+1}
    \\ &\quad = \iint_{0<t_1<\cdots<t_n\leq T}\, (f\circ \Phi_{\varepsilon} ^{-1})(t_1, \cdots, t_n) \prod_{i=1}^{n}\lambda^* (t_i;t_1,\cdots ,t_n) dt _1\cdots dt_n
    \\ & \qquad \times \int_T^{\infty}\lambda^* (t_{n+1};t_1,\cdots ,t_n)e^{-\int_{0}^{t_{n+1}} \lambda^* (s;t_1,\cdots ,t_n)\,ds}dt_{n+1}
    \\ & \quad= \iint_{0<t_1<\cdots<t_n\leq T}\, (f\circ \Phi_{\varepsilon} ^{-1})(t_1,\cdots,t_n) \varphi_n(t_1,\cdots, t_n)dt_1\cdots dt_n
    \\ & \quad= \iint_{0<u_1<\cdots<u_n\leq T}\, f(u_1,\cdots,u_n) (\varphi_n\circ \Phi_{\varepsilon})(u_1,\cdots,u_n) |\det J_{\Phi_{\varepsilon}}|du_1\cdots du_n
    \\ & \quad= \iint_{0<u_1<\cdots<u_n\leq T}\, f(u_1,\cdots,u_n) (\varphi_n\circ \Phi_{\varepsilon})(u_1,\cdots,u_n) 
    \\ & \qquad \times \prod_{i=1}^n (1+\varepsilon  m _{\varepsilon}(u_i))du_1\cdots du_n
     = \E\left[f(T_1,\cdots,T_n)\mathbf{1}_{\{N_T = n\}} Z_n^{\varepsilon}\right],
\end{align*}
  {where $\det J_{\Phi_{\varepsilon}}$ denotes the determinant of the Jacobian matrix of $\Phi_{\varepsilon}$.}\\
Let us emphasize that Assumption \ref{ass:lambda_mu} is used to ensure that: $\lambda^* (s;t_1,\cdots ,t_n) \geq \lambda_* > 0$. 
\end{proof}
\begin{Remark}
    This series converges in $L^1(\Omega)$ uniformly in $\varepsilon$ because for $f = 1$
    $$\mathbb{E} [G^\varepsilon] = \sum_{n=0}^{+\infty} \mathbb{E}[Z_n^\varepsilon \mathbf{1}_{\{N_T = n\}}] = \sum_{n=0}^{+\infty} \mathbb{E}^\varepsilon [\mathbf{1}_{\{N_T = n\}}] = 1.$$
\end{Remark}

\begin{Remark}
    In case of the Poisson process,  $\varphi_n$ is constant, equal to $\lambda^n$ and we have again the result obtained in \cite{Carlen1990} for standard Poisson processes
    $$\dfrac{d\P^{\varepsilon}}{d\P} =  \prod_{i=1}^{N_T}(1+\varepsilon  m _{\varepsilon}(T_i)).$$
\end{Remark}



We begin with a first result about the limits when $\varepsilon$ tends to $0$.

\begin{Proposition} \label{propo:G0=1}
    For any $n\in \mathbb{N}^*$, a.s. $$\lim_{\varepsilon \to 0} Z_n^\varepsilon =\lim_{\varepsilon \to 0} G^\varepsilon= 1. $$
\end{Proposition}

\begin{proof}
{We work on the set $\{N_T = n\}$. Hence for all  $i=1,\ldots,n$, $T_i \leq T$.}

    Firstly we have the almost surely convergences
    $$\forall i\in \{1, \cdots, n\}, \quad \varepsilon m_\varepsilon(T_i) \underset{\varepsilon \to 0}{\overset{a.s.}{\longrightarrow}} 0.$$
    Indeed we use \eqref{eq:def_m_eps} 
  to get
    \begin{eqnarray*}
        m(s) - m_\varepsilon(s) & = & m(s) - \widetilde{m}_\varepsilon(s) + \dfrac{1}{T} \int_0^T (m(r) - \widetilde{m}_\varepsilon(r)) dr.
    \end{eqnarray*}
    From the very definition of $\widetilde m_\varepsilon$, we deduce that:
    $$|m(T_i) - m_\varepsilon(T_i)|  \leq 2 \left| m(T_i) - \dfrac{1}{3\varepsilon} \right|\mathbf 1_{|m(T_i)| \geq \frac{1}{3\varepsilon}} +   \left|\dfrac{1}{T} \int_0^T (m(r) - \widetilde{m}_\varepsilon(r)) dr \right|.$$
    As $m(T_i) < \infty$ a.s., for $\varepsilon \in \mathbb{R}_+^*$ small enough, we have:
$$
        |m(T_i) - m_\varepsilon(T_i)| 
        \leq   \dfrac{1}{T} \int_0^T |m(r) - \widetilde{m}_\varepsilon(r)| dr
 \leq  \dfrac{1}{\sqrt{T}} \lVert m - \widetilde{m}_\varepsilon \rVert_\mathcal{H}. 
$$
Using Lemma \ref{lemma:convergence:meps}, we deduce that a.s. 
    $$\varepsilon m_\varepsilon(T_i) = \varepsilon m(T_i) + \varepsilon (m_\varepsilon(T_i) - m(T_i)) \underset{\varepsilon \to 0}{\longrightarrow} 0.$$
    So we have the convergence of the product in \eqref{eq:zneps}
    $$\prod_{i=1}^n (1+\varepsilon m_\varepsilon(T_i)) \underset{\varepsilon \to 0}{\overset{a.s.}{\longrightarrow}} 1.$$

    Secondly for the convergence of \eqref{eq:Phieps}, for a.e. $s\in [0,T]$,
    \begin{eqnarray*}
        |\widehat{m}(s) - \widehat{m}_\varepsilon(s)| & \leq & \int_0^T|m(r) - m_\varepsilon(r)| dr
        \leq \sqrt{T} \lVert m - m_\varepsilon\rVert_\mathcal{H} \underset{\varepsilon \to 0}{\longrightarrow} 0.
    \end{eqnarray*}
    In particular we have the uniform convergence on $[0,T]$ of $ \widehat{m}_\varepsilon$ to $\widehat m$. Then, for any $i\in \{1,\cdots,n\}$, 
    $$\varepsilon \widehat{m}_\varepsilon(T_i) = \varepsilon \widehat{m}(T_i) + \varepsilon (\widehat{m}_\varepsilon(T_i) - \widehat{m}(T_i)) \underset{\varepsilon \to 0}{\overset{a.s.}{\longrightarrow}} 0.$$
    Moreover
      {
    \begin{align*}
        & \varphi_n(\Phi_\varepsilon(T_1, \cdots, T_n)) = \varphi_n(T_1 + \varepsilon \widehat{m}_\varepsilon(T_1), \cdots, T_n+\varepsilon \widehat{m}_\varepsilon(T_n))
        \\ &\quad  =  \prod_{j=1}^n \lambda^*(T_j + \varepsilon \widehat{m}_\varepsilon(T_j); T_1 + \varepsilon \widehat{m}_\varepsilon(T_1), \cdots, T_n + \varepsilon \widehat{m}_\varepsilon(T_n))
        \\ & \qquad \times \exp\left(-\int_0^T \lambda^*(s;T_1 + \varepsilon \widehat{m}_\varepsilon (T_1), \cdots, T_n + \varepsilon \widehat{m}_\varepsilon(T_n)) ds\right)
        \\ & \quad =  \prod_{j=1}^n \left( \lambda_{T_j+\varepsilon \widehat{m}_\varepsilon (T_j)} +\gamma \left(  \sum_{i=1}^{j-1} \mu(T_j - T_i + \varepsilon(\widehat{m}_\varepsilon(T_j) - \widehat{m}_\varepsilon (T_i))) \right)\right)
        \\ & \qquad \times \exp \left( -\int_0^T \lambda_s ds-\int_0^T \gamma \left(  \sum_{i=1}^{n} \mu(s - T_i  - \varepsilon\widehat{m}_\varepsilon (T_i)) \mathbf 1_{\{s > T_i  + \varepsilon\widehat{m}_\varepsilon (T_i)\}} \right) ds \right).
    \end{align*}
    Thus, as the functions $\gamma$ and $\mu$ are continuous,}
    $$\varphi_n(\Phi_\varepsilon (T_1, \cdots, T_n)) \underset{\varepsilon \to 0}{\overset{a.s.}{\longrightarrow}} \varphi_n(T_1, \cdots, T_n).$$
    Therefore the a.s. convergence of $Z_n^\varepsilon$ is proved. The a.s. convergence of $G^\varepsilon$ follows.
\end{proof}

\subsection{Integration by parts in Bismut's way}\label{Section2.5}

\begin{Proposition} \label{prop:derivG}
Under our setting 
    \begin{equation*}
        \dfrac{\partial G^\varepsilon}{\partial \varepsilon}|_{\varepsilon = 0} = \int_{(0,T]} \left(\psi(m, s)  + \widehat{m}(s) \left(  \Gamma_1(s) + \Gamma_2(s) \right) + m(s)\right)) dN_s
    \end{equation*}
where $\psi$ is given by 
\begin{align} \nonumber
    \psi(m,s) & = \dfrac{1}{\lambda^*(s)} \widehat{m}(s)\lambda^\prime_s\\ \label{eq:def_psi}
    & + \dfrac{1}{\lambda^*(s)} \gamma' \left( \int_{(0,s)} \mu(s-t)dN_t \right)  \int_{(0,s)} (\widehat{m}(s) - \widehat{m}(t)) \mu'(s - t) dN_t
\end{align}
and $\Gamma_1$ and $\Gamma_2$ by:
\begin{align} \label{def:Gamma_1}
    \Gamma_1(s) & = \gamma \left(  \mu(0) + \int_{(0,s)} \mu(s-t)dN_t \right) - \gamma \left(\int_{(0,s)} \mu(s-t)dN_t \right) , \\  \label{def:Gamma_2}
    \Gamma_2(s) & = \int_s^T \gamma' \left(\int_{(0,u)} \mu(u-r)dN_r \right) \mu'(u-s)  du\\  \nonumber
    & = \int_0^{T-s} \gamma' \left(\int_{(0,v+s)} \mu(v+s-r)dN_r \right) \mu'(v)  dv.
\end{align}
\end{Proposition}

\begin{proof}
    Let $n \in \mathbb{N}^*$ and let's work on the event $\{N_T = n\}$. From \eqref{eq:zneps}, we have 
    \begin{eqnarray*}
        Z_n^\varepsilon & = & \dfrac{(\varphi_n \circ \Phi_\varepsilon)(T_1, \dots, T_n)}{\varphi_n(T_1, \dots, T_n)} \prod_{i=1}^n (1+\varepsilon m_\varepsilon(T_i)).
    \end{eqnarray*}
According to Proposition \ref{propo:G0=1}, $Z^\varepsilon_n$ almost surely converges to $Z^0_n = 1$ and for any $\varepsilon \in \mathbb{R}_+^*$
    \begin{eqnarray} \label{eq:growth_rate_Z_eps_n}
        \dfrac{Z_n^\varepsilon - 1}{\varepsilon} 
        & = & \dfrac{(\varphi_n \circ \Phi_\varepsilon)(T_1, \dots, T_n) - \varphi_n(T_1, \dots, T_n)}{\varepsilon \varphi_n(T_1, \dots, T_n)} \prod_{i=1}^n (1+\varepsilon m_\varepsilon(T_i)) \nonumber \\ 
        & & \quad +\dfrac{1}{\varepsilon} \left( \prod_{i=1}^n (1+\varepsilon m_\varepsilon(T_i)) - 1\right).
    \end{eqnarray}
For the first term
        \begin{align*}
            & \dfrac{(\varphi_n \circ \Phi_\varepsilon)(T_1, \dots, T_n) - \varphi_n(T_1, \dots, T_n)}{\varepsilon \varphi_n(T_1, \dots, T_n)}
            \\ &\quad  = \dfrac{\varphi_n(T_1 + \varepsilon \widehat{m}_\varepsilon(T_1), \dots, T_n + \varepsilon \widehat{m}_\varepsilon(T_n)) - \varphi_n(T_1, \dots, T_n)}{\varepsilon \varphi_n(T_1, \dots, T_n)}
            \\ & \quad =  \sum_{i=1}^n \int_0^1 \dfrac{\partial \varphi_n}{\partial t_i}(T_1 + \alpha \varepsilon \widehat{m}_\varepsilon(T_1), \dots, T_n + \alpha \varepsilon \widehat{m}_\varepsilon(T_n)) \dfrac{\varepsilon \widehat{m}_\varepsilon(T_i)}{\varepsilon \varphi_n(T_1, \cdots, T_n)} d\alpha
            \\ &\quad  =  \sum_{i=1}^n \dfrac{\widehat{m}_\varepsilon(T_i)}{\varphi_n(T_1, \cdots, T_n)}  \int_0^1 \dfrac{\partial \varphi_n}{\partial t_i}(T_1 + \alpha \varepsilon \widehat{m}_\varepsilon(T_1), \dots, T_n + \alpha \varepsilon \widehat{m}_\varepsilon(T_n)) d\alpha
            \\ &\quad  \underset{\varepsilon \to 0}{\overset{a.s.}{\longrightarrow}}  \sum_{i=1}^n \dfrac{\widehat{m}(T_i)}{\varphi_n(T_1, \dots, T_n)} \dfrac{\partial \varphi_n}{\partial t_i} (T_1, \dots, T_n)
        \end{align*}
where the almost sure convergence is justified by the fact that for any $i\in \{1, \ldots, n\}$, $ \widehat{m}_\varepsilon(T_i) \underset{\varepsilon \to 0}{\overset{a.s.}{\longrightarrow}} \widehat{m}(T_i)$ and  $\dfrac{\partial \varphi_n}{\partial t_i}$ is bounded because $\gamma$ and $\mu$ admit a bounded derivative. 
Moreover as in the proof of Proposition \ref{propo:G0=1},
        $$\prod_{i=1}^n (1+\varepsilon m_\varepsilon(T_i)) \underset{\varepsilon \to 0}{\overset{a.s.}{\longrightarrow}} 1.$$
For the second term
        \begin{align*}
           &  \dfrac{\prod_{i=1}^n (1+\varepsilon m_\varepsilon(T_i)) - 1}{\varepsilon} \\
           & \quad = \sum_{i=1}^n m_\varepsilon(T_i) + \varepsilon \sum_{1\leq i< j \leq n} m_\varepsilon(T_i) m_\varepsilon (T_j)+ \cdots + \varepsilon^{n-1} \prod_{i=1}^n m_\varepsilon (T_i)
        \end{align*}
        with, for any $i\in \{1, \dots, n\}$, $m_\varepsilon(T_i) \underset{\varepsilon\to 0}{\overset{a.s.}{\longrightarrow}} m(T_i)$. Thus
        $$\dfrac{\prod_{i=1}^n (1+\varepsilon m_\varepsilon(T_i)) - 1}{\varepsilon} \underset{\varepsilon \to 0}{\overset{a.s.}{\longrightarrow}} \sum_{i=1}^n m(T_i).$$
    Therefore
    \begin{equation} \label{eq:convergence_Z_n_epsilon}
    \dfrac{Z_n^\varepsilon - Z_n^0}{\varepsilon} \underset{\varepsilon \to 0}{\overset{a.s.}{\longrightarrow}} \sum_{i=1}^n \dfrac{\widehat{m}(T_i)}{\varphi_n(T_1, \dots, T_n)} \dfrac{\partial \varphi_n}{\partial t_i} (T_1, \dots, T_n) + \sum_{i=1}^n m(T_i)
    \end{equation}
    with, from \eqref{eq:phin}, for any $i_0\in \{1, \dots, n\}$ and $0< t_1 < \dots < t_n \leq T$,
    \begin{align*}
        &\dfrac{1}{\varphi_n(t_1, \dots, t_n)} \dfrac{\partial \varphi_n}{\partial t_{i_0}} (t_1, \dots, t_n) =  \dfrac{\partial}{\partial t_{i_0}} \ln(\varphi_n(t_1, \dots, t_n))
        \\ & \quad =  \sum_{i=1}^n \dfrac{\partial}{\partial t_{i_0}}\ln(\lambda^*(t_i; t_1, \dots, t_n)) - \dfrac{\partial}{\partial t_{i_0}} \left(\int_0^T \lambda^*(s;t_1, \dots, t_n) ds\right)
        \\ & \quad =  \dfrac{1}{\lambda^*(t_{i_0};t_1, \dots t_{i_0-1})} \dfrac{\partial \lambda^*}{\partial t_{i_0}} (t_{i_0}; t_1, \dots, t_{i_0-1}) \\
        & \quad + \sum_{i=i_0+1}^n \dfrac{1}{\lambda^*(t_i;t_1, \dots t_{i-1})} \dfrac{\partial \lambda^*}{\partial t_{i_0}} (t_i; t_1, \dots, t_{i-1}) 
     - \dfrac{\partial}{\partial t_{i_0}} \left(\int_0^T \lambda^*(s;t_1, \dots, t_n) ds\right). 
    \end{align*}
From \eqref{def:lambda_etoile}, evoke that for any $i\in \{i_0+1, \ldots, n\}$ and $s\in [0,T]$,
  {
    \begin{eqnarray*}
        \lambda^*(t_{i_0};t_1, \dots, t_{i_0-1}) & = & \lambda_{t_{i_0}}+\gamma \left( \sum_{j=1}^{i_0-1} \mu(t_{i_0}-t_j)\right),
        \\ \lambda^*(t_i;t_1, \dots, t_{i-1}) & = & \lambda_{t_i}+\gamma \left(\sum_{j=1}^{i-1} \mu(t_i-t_j) \right),
        \\ \lambda^* (s;t_1,\cdots ,t_n) & = &\lambda_s+ \gamma \left(\sum_{j=1}^n \mu(s-t_j) 1_{\{s > t_j\}}\right).
    \end{eqnarray*}
}
Thus
  {
    \begin{eqnarray*}
        \dfrac{\partial \lambda^*}{\partial t_{i_0}}(t_{i_0};t_1, \dots, t_{i_0 - 1}) & = &\lambda^\prime_{t_{i_0}}+ \left( \sum_{j=1}^{i_0 - 1} \mu'(t_{i_0} - t_j) \right) \gamma' \left( \sum_{j=1}^{i_0-1} \mu(t_{i_0}-t_j)\right), 
        \\ \dfrac{\partial \lambda^*}{\partial t_{i_0}}(t_i;t_1, \dots, t_{i - 1}) & = & - \mu'(t_i -t_{i_0}) \gamma' \left(\sum_{j=1}^{i-1} \mu(t_i-t_j) \right),
    \end{eqnarray*}
    and 
        \begin{align*}
        & \int_0^T \lambda^*(s;t_1, \dots, t_n) ds  = \int_{t_{i_0}}^T \lambda^*(s;t_1, \dots, t_n) ds +  \int_0^{t_{i_0}} \lambda^*(s;t_1, \dots, t_n) ds
        \\ & \quad = \int_0^T \lambda_s ds+\int_{t_{i_0}}^T  \gamma \left(\sum_{j=1}^n \mu(s-t_j) 1_{\{s > t_j\}}\right) ds +  \int_0^{t_{i_0}} \gamma \left(\sum_{j=1}^{i_0-1} \mu(s-t_j) 1_{\{s > t_j\}}\right) ds.
\end{align*}
Thus 
    \begin{align*}  \dfrac{\partial}{\partial t_{i_0}}\int_0^T \lambda^*(s;t_1, \dots, t_n) ds & = \gamma \left(\sum_{j=1}^{i_0-1} \mu(t_{i_0}-t_j)\right) - \gamma \left(\sum_{j=1}^{i_0-1} \mu(t_{i_0}-t_j) + \mu(0) \right) \\
    & - \int_{t_{i_0}}^T  \gamma' \left(\sum_{j=1}^n \mu(s-t_j) 1_{\{s > t_j\}}\right) \mu'(s-t_{i_0}) ds.
    \end{align*}
Note that in the linear case $\gamma(x)=x$:
  \begin{align*}  \dfrac{\partial}{\partial t_{i_0}}\int_0^T \lambda^*(s;t_1, \dots, t_n) ds & = - \mu(0) - \int_{t_{i_0}}^T  \mu'(s-t_{i_0}) ds = -\mu(T-t_{i_0}).
    \end{align*}
}
    Then
    \begin{eqnarray} \label{eq:compute_derivative}
        && \sum_{i_0=1}^n \dfrac{\widehat{m}(T_{i_0})}{\varphi_n(T_1, \dots, T_n)} \dfrac{\partial \varphi_n}{\partial t_{i_0}}(T_1, \dots, T_n)
        \\  \nonumber & = & \sum_{i_0=1}^n  \dfrac{\widehat{m}(T_{i_0})}{\lambda^*(T_{i_0};T_1, \dots, T_n)}\left( \lambda^\prime_{T_{i_0}}+ \left( \sum_{j=1}^{i_0 - 1} \mu'(T_{i_0} - T_j) \right)\gamma' \left( \sum_{j=1}^{i_0-1} \mu(T_{i_0}-T_j)\right)\right) \\ \nonumber
        && - \sum_{i_0 = 1}^n \widehat{m}(T_{i_0}) \sum_{i=i_0+1}^n \dfrac{\mu'(T_i - T_{i_0})}{\lambda^*(T_i;T_1, \dots, T_n)} \gamma' \left( \sum_{j=1}^{i-1} \mu(T_{i}-T_j)\right) \\  \nonumber && + \sum_{i_0 = 1}^n \widehat{m}(T_{i_0}) \left(\gamma \left(\sum_{j=1}^{i_0-1} \mu(T_{i_0}-T_j) + \mu(0) \right) - \gamma \left(\sum_{j=1}^{i_0-1} \mu(T_{i_0}-T_j)\right) \right) 
        \\  \nonumber && + \sum_{i_0 = 1}^n \widehat{m}(T_{i_0}) \int_{T_{i_0}}^T  \gamma' \left(\sum_{j=1}^n \mu(s-T_j) 1_{\{s > T_j\}}\right) \mu'(s-T_{i_0}) ds .
     \end{eqnarray}
    Then we would like the same convergence result \eqref{eq:convergence_Z_n_epsilon} for $\dfrac{\partial G^\varepsilon}{\partial \varepsilon} |_{\varepsilon = 0}$. We have, according to Proposition \ref{propo:G0=1},
    \begin{eqnarray*}
        \dfrac{\partial G^\varepsilon}{\partial \varepsilon}_{|\varepsilon = 0} \mathbf{1}_{\{N_T = n\}} & = & \lim_{\varepsilon \to 0}\dfrac{G^\varepsilon - 1}{\varepsilon} \mathbf{1}_{\{N_T = n\}} 
= \lim_{\varepsilon \to 0}\dfrac{G^\varepsilon \mathbf{1}_{\{N_T = n\}} - \mathbf{1}_{\{N_T = n\}}}{\varepsilon}
        \\ & = & \lim_{\varepsilon \to 0} \dfrac{Z_n^\varepsilon \mathbf{1}_{\{N_T = n\}} - \mathbf{1}_{\{N_T = n\}}}{\varepsilon}
= \lim_{\varepsilon \to 0} \dfrac{Z_n^\varepsilon - 1}{\varepsilon} \mathbf{1}_{\{N_T = n\}}
        \\ & = & \dfrac{\partial Z_n^\varepsilon}{\partial \varepsilon}|_{\varepsilon = 0} \mathbf{1}_{\{N_T = n\}}.
    \end{eqnarray*}
    Therefore
    \begin{eqnarray*}
        \dfrac{\partial G^\varepsilon}{\partial \varepsilon}|_{\varepsilon = 0} & = &  \sum_{n=1}^{\infty} \dfrac{\partial G^\varepsilon}{\partial \varepsilon}|_{\varepsilon = 0} \mathbf{1}_{\{N_T = n\}} = \sum_{n=1}^{\infty} \dfrac{\partial Z_n^\varepsilon}{\partial \varepsilon}|_{\varepsilon = 0} \mathbf{1}_{\{N_T = n\}}.
        \\ & = & \sum_{i_0=1}^{N_T} \left( \dfrac{\widehat{m}(T_{i_0})}{\varphi_{N_T}(T_1, \dots, T_{N_T})} \dfrac{\partial \varphi_{N_T}}{\partial t_{i_0}}(T_1, \dots, T_{N_T}) + m(T_{i_0})\right). 
\end{eqnarray*}
Now in \eqref{eq:compute_derivative}, the first two sums of the right-hand side can be written:
$$\sum_{i_0=1}^{N_T}  \dfrac{1}{\lambda^*(T_{i_0};T_1, \dots, T_{N_T})}\widehat{m}(T_{i_0})\lambda^\prime_{T_{i_0}} = \int_0^T \dfrac{\widehat{m}(s)\lambda^\prime_{s}}{\lambda^*(s)} dN_s$$
and
\begin{eqnarray*}
&&\sum_{i_0=1}^{N_T}  \dfrac{1}{\lambda^*(T_{i_0};T_1, \dots, T_{N_T})} \left(  \sum_{j=1}^{i_0 - 1} (\widehat{m}(T_{i_0}) - \widehat{m}(T_j)) \mu'(T_{i_0} - T_j) \right) \gamma' \left( \sum_{j=1}^{i_0-1} \mu(T_{i_0}-T_j)\right) \\
&&=\int_0^T \dfrac{1}{\lambda^*(s)} \left(  \int_{(0,s)} (\widehat{m}(s) - \widehat{m}(t)) \mu'(s - t) dN_t \right) \gamma' \left( \int_{(0,s)} \mu(s-t)dN_t \right) dN_s.
\end{eqnarray*}
The third term is equal to:
$$\int_0^T \widehat m(s)\left(\gamma \left(  \mu(0) + \int_{(0,s)} \mu(s-t)dN_t \right) - \gamma \left(\int_{(0,s)} \mu(s-t)dN_t \right) \right) dN_s. $$
The last one can be formulated as:
$$\int_0^T \widehat m(s)\left( \int_s^T \gamma' \left(\int_{(0,u)} \mu(u-r)dN_r \right) \mu'(u-s) du \right) dN_s. $$
Therefore
    \begin{eqnarray*}
        \dfrac{\partial G^\varepsilon}{\partial \varepsilon}|_{\varepsilon = 0} & = & \int_{(0,T]} (\psi(m, s)  + \widehat{m}(s) \left(  \Gamma_1(s) + \Gamma_2(s) \right) + m(s)) dN_s
    \end{eqnarray*}
    where $\psi$, $\Gamma_1$ and $\Gamma_2$ are given by \eqref{eq:def_psi},
\eqref{def:Gamma_1} and \eqref{def:Gamma_2}. 
\end{proof}

\begin{Remark} \label{rem:deriv_linear_case}
Let us emphasize that $\psi (m,\cdot )$ and $\Gamma_1$ are predictable, but $\Gamma_2$ is $\mathcal F_T$-measurable and that in the linear case $\gamma(x) = x$, $\Gamma_1$ and $\Gamma_2$ are deterministic with 
$$\Gamma_1(s) + \Gamma_2(s) = \mu(T-s).$$
\end{Remark}
\begin{Theorem} \label{th:DmF}
    For any $F \in \mathcal{S}$,
    $$\mathbb{E}[D_m F] = \mathbb{E}\left[ \dfrac{\partial G^\varepsilon}{\partial \varepsilon}|_{\varepsilon = 0} F \right].$$
\end{Theorem}

\begin{proof}
    We consider $F = f_n(T_1, \cdots, T_n) 1_{\{N_T = n\}} \in \mathcal{S}$. Then, as the vectors $(\tau_\varepsilon^{-1}(T_1), \cdots, \tau_\varepsilon^{-1}(T_n))$ and $(T_1, \cdots, T_n)$ are in the compact set $[0,T]^n$ and the function $f_n$ is smooth, there exists a constant $C = C(f_n,n,T) \in \mathbb{R}_+^*$ such that
    \begin{eqnarray*}
        \left|\dfrac{\mathcal{T}_\varepsilon F - F}{\varepsilon}\right| & = & \dfrac{|f_n(\tau_\varepsilon^{-1}(T_1), \cdots, \tau_\varepsilon^{-1}(T_n)) - f_n(T_1, \cdots, T_n)|}{\varepsilon} \mathbf{1}_{\{N_T = n\}}
        \\ & \leq & C \dfrac{\lVert (\tau_\varepsilon^{-1}(T_1), \cdots, \tau_\varepsilon^{-1}(T_n)) - (T_1, \cdots, T_n) \rVert}{\varepsilon} \mathbf{1}_{\{N_T = n\}}
        \\ & = & C \dfrac{\lVert (\tau_\varepsilon^{-1}(T_1) - T_1, \cdots, \tau_\varepsilon^{-1}(T_n) - T_n)\rVert}{\varepsilon} \mathbf{1}_{\{N_T = n\}}
        \\ & = & C \sup_{1\leq i\leq n} \dfrac{|\tau_\varepsilon^{-1}(T_i) - T_i|}{\varepsilon} \mathbf{1}_{\{N_T = n\}}
        \\ & \overset{(*)}{\leq} & C \int_0^T |m_\varepsilon(s)|ds \mathbf{1}_{\{N_T = n\}}
        \\ & \leq & C \int_0^T |m_\varepsilon(s) - m(s)|ds \mathbf{1}_{\{N_T = n\}} + C \int_0^T |m(s)| ds \mathbf{1}_{N_T = n}
        \\ & \leq & C^2 \1{N_T = n} + C \int_0^T |m(s)| ds \mathbf{1}_{\{N_T = n\}}
    \end{eqnarray*}
    where the last inequality is true for $\varepsilon$ small enough because $m_\varepsilon \underset{\varepsilon \to 0}{\overset{\mathcal{H}}{\longrightarrow}} m$ according to Lemma \ref{lemma:convergence:meps}, and the inequality $(*)$ is due to the equality
    $$T_i = \tau_\varepsilon(\tau_\varepsilon^{-1}(T_i)) = \tau_\varepsilon^{-1}(T_i) + \varepsilon\int_0^{\tau_\varepsilon^{-1}(T_i)} m_\varepsilon(s) ds.$$
    Thus, by dominated convergence theorem and since $\mathbb{E}[\mathcal{T}_\varepsilon F]  = \mathbb{E}[G^\varepsilon F]$,
    \begin{eqnarray*}
        \mathbb{E}[D_m F] & = & \mathbb{E}\left[ \lim_{\varepsilon \to 0} \dfrac{\mathcal{T}_\varepsilon F - F}{\varepsilon} \right]
      = \lim_{\varepsilon \to 0} \mathbb{E} \left[ \dfrac{\mathcal{T}_\varepsilon F - F}{\varepsilon} \right]
   = \lim_{\varepsilon \to 0} \mathbb{E}\left[\dfrac{G^\varepsilon - 1}{\varepsilon} F\right].
    \end{eqnarray*}
    Hence, as $F = f_n(T_1, \cdots, T_n) \mathbf{1}_{\{N_T = n\}}$,
    \begin{eqnarray*}
        \mathbb{E}[D_m F] & = & \lim_{\varepsilon \to 0} \mathbb{E}\left[\dfrac{G^\varepsilon - 1}{\varepsilon} f_n(T_1, \cdots, T_n) \mathbf{1}_{\{N_T = n\}}\right]
        \\ & = & \lim_{\varepsilon \to 0} \mathbb{E}\left[\dfrac{G^\varepsilon \mathbf{1}_{N_T = n} - 1}{\varepsilon} f_n(T_1, \cdots, T_n) \mathbf{1}_{\{N_T = n\}}\right]
        \\ & = & \lim_{\varepsilon \to 0} \mathbb{E}\left[\dfrac{Z_n^\varepsilon - 1}{\varepsilon} f_n(T_1, \cdots, T_n) \mathbf{1}_{\{N_T = n\}}\right]
  = \lim_{\varepsilon \to 0} \mathbb{E}\left[\dfrac{Z_n^\varepsilon - 1}{\varepsilon}F\right].
    \end{eqnarray*}
    Let us come back to the  expression \eqref{eq:growth_rate_Z_eps_n} of the growth rate of $Z^\varepsilon_n$. 
        For the last term in \eqref{eq:growth_rate_Z_eps_n}, we have 
        $$\dfrac{1}{\varepsilon} \left[ \prod_{i=1}^n(1+\varepsilon  m_{\varepsilon}(T_i)) - 1\right] \underset{\varepsilon \to 0}{\overset{a.s.}{\longrightarrow}} \sum_{i=1}^n m(T_i)$$
        and
        \begin{align*}
          \left|\dfrac{1}{\varepsilon} \left[ \prod_{i=1}^n(1+\varepsilon  m_{\varepsilon}(T_i)) - 1\right] \right| & \leq  \sum_{i=1}^n |m_\varepsilon(T_i)| + \varepsilon \sum_{1\leq i< j \leq n} |m_\varepsilon(T_i) m_\varepsilon(T_j)|+ \cdots \\
          & + \varepsilon^{n-1} \prod_{i=1}^n |m_\varepsilon(T_i)|.
        \end{align*}
We know that $(T_1, \cdots, T_n)$ knowing $\{N_T = n\}$ admits a density given by 
\eqref{eq:cond_dens_jump_times} in Lemma \ref{lemma:density} and that this density is bounded (see \eqref{eq:bound_dens_jump_times}). 
                { 
        Therefore, for any $I\subset \{1, ...,n\}$, with cardinal denoted $|I|$:
        \begin{eqnarray*}
            \mathbb{E}\left[\left| \prod_{i\in I} m_\varepsilon(T_i) \right| \right] & \leq & \dfrac{1}{\mathbb{P}(N_T = n)} \left(\lambda^T + n a\lVert \mu \rVert_{\infty}\right)^n
            \\ && \times \int_{[0,T]^n} \left| \prod_{i\in I} m_\varepsilon(t_i) \right| \mathbf{1}_{\{0<t_1<\cdots <t_{n} < T\}} dt
            \\ & \leq & \dfrac{1}{\mathbb{P}(N_T = n)} \left(\lambda^T + n a\lVert \mu \rVert_{\infty}\right)^n \int_{[0,T]^n} \prod_{i\in I} |m_\varepsilon(t_i)| dt
            \\ & = & \dfrac{1}{\mathbb{P}(N_T = n)} \left(\lambda^T + n a\lVert \mu \rVert_{\infty}\right)^n T^{n-|I|} \prod_{i\in I} \int_0^T |m_\varepsilon(t_i)| dt_i
            \\ & = & \dfrac{1}{\mathbb{P}(N_T = n)} \left(\lambda^T + n a\lVert \mu \rVert_{\infty}\right)^n T^{n-|I|} \left(\int_0^T |m_\varepsilon(t)| dt \right)^{|I|}
            \\ & \leq & \dfrac{1}{\mathbb{P}(N_T = n)} \left(\lambda^T + n a\lVert \mu \rVert_{\infty}\right)^n T^{n-|I|} \left(\int_0^T |m(t)| dt + 1 \right)^{|I|}.
        \end{eqnarray*} }
        The last inequality is justified by the convergence proved in Lemma \ref{lemma:convergence:meps}. 
        {  This proves that conditionally to $\{ N_T =n \}$, $\dfrac{1}{\varepsilon} \left[ \prod_{i=1}^n(1+\varepsilon  m_{\varepsilon}(T_i)) - 1\right]$ converges in $L^1(\Omega)$ to $ \sum_{i=1}^n m(T_i)$.}

 {  For the first term in \eqref{eq:growth_rate_Z_eps_n}, since $|\varepsilon m_\varepsilon (\cdot)|\leq 5/3$
        $$\left| \prod_{i=1}^n (1+\varepsilon m_\varepsilon (T_i)) \right| \leq \left(1+\frac53 \right)^n$$}
        and
        \begin{align*}
            & \dfrac{(\varphi_n\circ \Phi_{\varepsilon})(T_1,\cdots,T_n)-\varphi_n (T_1,\cdots,T_n)}{\varepsilon} \\
            = & \sum_{k=1}^n \int_0^1 \dfrac{\partial \varphi_n}{\partial t_k} \left( T_1,\ldots,T_{k-1},T_k + \alpha \varepsilon \widehat m_\varepsilon(T_k),T_{k+1} + \varepsilon \widehat m_\varepsilon(T_{k+1}),\ldots\right) d\alpha \ \widehat m_\varepsilon(T_k).
        \end{align*}
        Let us define 
        $$\psi_n(t_1,\cdots,t_n)=\prod_{i=1}^{n}\lambda^* (t_i;t_1,\cdots ,t_n).$$
        Thus
        $$\varphi_n(t_1, \cdots, t_n) = \psi_n(t_1,\cdots,t_n) e^{- \int_0^T\lambda^*(s;t_1,\cdots,t_n) ds}.$$
        Therefore, as $\mu$ is differentiable,
        \begin{align*}\label{eqCVD}
            & \dfrac{1}{ \varphi_n (T_1,\cdots,T_n)}  \dfrac{\partial \varphi_n}{\partial t_k} \left( T_1,\ldots,T_{k-1},T_k + \alpha \varepsilon \widehat m_\varepsilon(T_k),T_{k+1} + \varepsilon \widehat m_\varepsilon(T_{k+1}),\ldots\right) \\
            = & \left[ \dfrac{1}{ \psi_n(T_1,\cdots,T_n)} \dfrac{\partial \psi_n}{\partial t_k} \left( T_1,\ldots,T_{k-1},T_k + \alpha \varepsilon \widehat m_\varepsilon(T_k),T_{k+1} + \varepsilon \widehat m_\varepsilon(T_{k+1}),\ldots\right) \right.\\
            & \qquad \left. -  \dfrac{\partial }{\partial t_k} \int_{0}^{T} \lambda^* (s; T_1,\ldots,T_{k-1},T_k + \alpha \varepsilon \widehat m_\varepsilon(T_k),T_{k+1} + \varepsilon \widehat m_\varepsilon(T_{k+1}),\ldots)ds\right]
        \end{align*}
        with, for the first term { 
        $$\psi_n(t_1, \cdots, t_n) = \prod_{i=1}^n \left( \lambda_{t_i} + \gamma \left(\sum_{j=1}^{i-1} \mu(t_i - t_j) 1_{\{t_i > t_j\}} \right)\right) .$$
        Thus
        $$|\psi_n(t_1, \cdots, t_n)| \geq (\lambda^T)^n$$ }
        {  Using the regularity of $\gamma$, $\mu$ and $\lambda$, we can verify that there exist constants that we denote $C$ even if they may change from line to line, depending only on $n,T,\lambda^T, \parallel \mu \parallel_\infty,\parallel \mu^\prime \parallel_\infty  , \parallel \gamma \parallel_\infty $ and $\parallel \gamma^\prime \parallel_\infty$  such that  for any $0 < t_1 < \cdots < t_n < T$ and $k\in\{ 1,\cdots ,n\}$:
        $$\left| \dfrac{\partial \psi_n}{\partial t_k}(t_1, \cdots, t_n)\right| \leq C,\qquad \left|  \dfrac{\partial \lambda^*}{\partial t_k} (t_1, \cdots, t_n) \right| \leq C.$$
    This yields immediately:
    $$ \left\vert \dfrac{(\varphi_n \circ \Phi_\varepsilon)(T_1, \dots, T_n) - \varphi_n(T_1, \dots, T_n)}{\varepsilon \varphi_n(T_1, \dots, T_n)} \prod_{i=1}^n (1+\varepsilon m_\varepsilon(T_i))\right\vert\leq C,$$
    hence the convergence in $L^1(\Omega)$ of $$\dfrac{(\varphi_n \circ \Phi_\varepsilon)(T_1, \dots, T_n) - \varphi_n(T_1, \dots, T_n)}{\varepsilon \varphi_n(T_1, \dots, T_n)} \prod_{i=1}^n (1+\varepsilon m_\varepsilon(T_i)).$$
        }
{  So,  we have proved that conditionally to $\{N_T =n\}$, $\dfrac{Z_n^\varepsilon - 1}{\varepsilon}$ converges to $\dfrac{\partial G^\varepsilon}{\partial \varepsilon}$ in $L^1(\Omega)$, thus:}

$$
        \mathbb{E}[D_m F] = \mathbb{E}\left[\lim_{\varepsilon \to 0}\dfrac{Z_n^\varepsilon - 1}{\varepsilon}F\right]
    = \mathbb{E}\left[ \dfrac{\partial G^\varepsilon}{\partial \varepsilon} F \right].
$$
    Then, by linearity, we deduce the same equality for any $F\in \mathcal{S}$.
\end{proof}

\begin{Remark}
    Taking  $F = 1$, since $D_m F = 0$ we get:
    $\E\left[\dfrac{\partial G^{\varepsilon}}{\partial \varepsilon}|_{\varepsilon = 0}\right] = 0.$
\end{Remark}
{Note that we can use this property to derive the functional equation satisfied by $s \mapsto \mathbb E \lambda^*(s)$ in the linear case when $\gamma(x) = x$. Indeed with Proposition \ref{prop:derivG} and Remark \ref{rem:deriv_linear_case}, we have:}
   \begin{align*}
        & \E\left[\dfrac{\partial G^{\varepsilon}}{\partial \varepsilon}|_{\varepsilon = 0}\right] =  \E \int_0^T (\psi(m,s) + \widehat{m}(s) \mu(T-s) +  m(s)) \lambda^*(s) ds 
        \\ & \quad =  \E  \int_0^T \left(\widehat{m}(s) \lambda_s' + \int_{(0,s)}(\widehat{m}(s) - \widehat{m}(t)) \mu'(s-t) dN_t + (m(s) + \widehat{m}(s) \mu(T-s))\lambda^*(s) \right) ds
        \\ &\quad =  \int_0^T \int_0^s (\widehat{m}(s) - \widehat{m}(t)) \mu'(s-t) \E[\lambda^*(t)] dt ds + \int_0^T (m(s) + \widehat{m}(s) \mu(T-s))\E[\lambda^*(s)] ds
        \\ & \qquad +\int_0^T \widehat{m}(s) \lambda_s' ds
        \\ & \quad =  \int_0^T \int_0^s (\widehat{m}(s) - \widehat{m}(t)) \mu'(s-t) g(t) dt ds + \int_0^T (m(s) + \widehat{m}(s) \mu(T-s)) g(s) ds\\
        & \qquad - \int_0^T m(s) \lambda_s ds,
\end{align*}
with $g(s) = \mathbb{E}[\lambda^*(s)]$, an integration by parts and with the property that $\widehat m(T)=\widehat m(0)=0$. Thus
  \begin{align*}      
        &  \E\left[\dfrac{\partial G^{\varepsilon}}{\partial \varepsilon}|_{\varepsilon = 0}\right] = 
        \int_0^T \left( \left[(\widehat{m}(s) - \widehat{m}(t)) \mu(s-t) \right]_t^T - \int_t^T m(s)\mu(s-t) ds \right) g(t) dt
        \\ & \qquad + \int_0^T (m(s) + \widehat{m}(s) \mu(T-s)) g(s) ds  - \int_0^T m(s) \lambda_s ds
        \\ & \quad =  - \int_0^T \widehat{m}(t) \mu(T-t) g(t) dt - \int_0^T m(s) \left(\int_0^s \mu(s-t) g(t) dt \right)ds 
        \\ & \qquad + \int_0^T (m(s) + \widehat{m}(s) \mu(T-s)) g(s) ds - \int_0^T m(s) \lambda_s ds\\
        & \quad =  \int_0^T m(s)  \left[ g(s) - \left(\int_0^s \mu(s-t) g(t) dt \right) - \lambda_s \right]ds .
    \end{align*}
Therefore $g$ satisfies the known equality (see  \cite{laub:taimre:pollett}):
    $$g(s) = \lambda_s + \int_0^s \mu(s-t) g(t) dt.$$

\subsection{Directional Dirichlet space}
{ The first step consists in closing operator $D_m$ for each $m\in \mathcal{H}$.}

\begin{Proposition} \label{propo:closable}
     Let $m$ in $\mathcal{H}$. The quadratic bilinear form on $L^2(\Omega)$, $(\mathcal{S}, \mathcal{E}_m)$ defined by 
    $$\forall X,Y\in\mathcal{S},\quad  \mathcal{E}_m (X,Y)=\E [D_m X D_m Y],$$
    is closable. We denote by $(\mathbb{D}_m^{1,2},\mathcal{E}_m) $ its closed extension. As a consequence, $D_m$ is also closable and we still denote by $D_m$ its extension which is well-defined on the whole space $\mathbb{D}_m^{1,2}$. {  Moreover, $(\mathbb{D}_m^{1,2},\mathcal{E}_m) $ is a local Dirichlet form    and Theorem \ref{th:DmF} remains valid for functions in $\mathbb{D}_m^{1,2}$.}
\end{Proposition}

\begin{proof}
    Let $(X_n)_{n\in \mathbb{N}}$ be a sequence in $\mathcal{S}$ converging to $0$ in $L^2(\Omega)$ and such that 
    $$\lim_{n,k\rightarrow +\infty} \mathcal{E}_m (X_n -X_k ) = \lim_{n,k \to +\infty} \mathbb{E}[(D_m X_n - D_m X_k)^2] =0.$$
    Thus $(D_m X_n )_{n\in \mathbb{N}}$ is a Cauchy sequence in $L^2(\Omega)$, so it converges to an element $Z$ in $L^2(\Omega)$. Then, let $Y$ be in $\mathcal{S}$, we have by integration by part formula:
    \begin{eqnarray*}
        \E [D_m X_nY]&=&\E [D_m (X_n Y)]-\E [X_n D_m Y]\\
        &=&\E \left[X_n Y \dfrac{\partial G^\varepsilon}{\partial \varepsilon}|_{\varepsilon = 0}\right]-\E [X_n D_m Y].
    \end{eqnarray*}
    The last equality comes from Theorem \ref{th:DmF}. 
    Since $Y\dfrac{\partial G^\varepsilon}{\partial \varepsilon}|_{\varepsilon = 0}$ and $D_m Y$ belong to $L^2(\Omega)$, we get, by dominated convergence theorem,
    $$\forall Y\in \mathcal{S},\quad \mathbb{E}[ZY] = \lim_{n\rightarrow +\infty}\E [D_m X_nY] = 0.$$
    Hence $Z=0$ by density. We deduce thanks to \cite[Proposition 1.3.2]{bouleau:hirsch} that $(\mathcal{S}, \mathcal{E}_m)$ is closable. \\
    As a consequence, for any
    $X\in \mathbb{D}_m^{1,2}$, there exists a sequence $(X_n )_{n\in \mathbb{N}} \in \mathcal{S}$ converging to $X$ in $\mathbb{D}_m^{1,2}$ since for all $n,k\in\N$
    $$\mathcal{E}_m (X_n-X_k)=\E [|D_m X_n - D_m X_k|^2 ]$$
    we deduce that $(D_m X_n)_{n\in \mathbb{N}}$ is a Cauchy sequence in $L^2$ hence converges to  an element that we still denote $D_m X$ and defines in a unique way the extension of $D_m$ to $\mathbb{D}_m^{1,2}$.\\
    As a consequence of Theorem 3.1.1 in \cite{FukushimaOshimaTakeda+1994}  $(\mathbb{D}_m^{1,2},\mathcal{E}_m) $ is a    Dirichlet form. Finally, 
    by  an approximation procedure, it is clear that point 2. of Proposition \ref{functCalculus} holds    {with $\Phi$, $C^1$ with }bounded derivatives and $F_1 , \cdots ,F_d$ in $\mathbb{D}_m^{1,2}$ which implies the local property. \\
      {Similarly, equality in Theorem \ref{th:DmF} is obtained by density.}
\end{proof}


\section{The local Dirichlet form}

We would like to define an operator $D$ with domain $\mathbb{D}^{1,2} \subset L^2(\Omega)$ and taking values in $L^2(\Omega;\mathcal{H})$ such that 
$$\forall F \in \mathbb{D}^{1,2}, m\in \mathcal{H}, \quad D_m F = \langle DF, m\rangle_\mathcal{H} = \int_0^T D_t F m(t) dt.$$

\subsection{Definition using a Hilbert basis}

Let $(m_i)_{i\in \mathbb{N}}$ be a Hilbert basis of the space $\mathcal{H}$. Then every function $m\in \mathcal{H}$ can be expressed as
$$m = \sum_{i=0}^{+\infty} \langle m, m_i\rangle_\mathcal{H} m_i.$$
We now set
$$\mathbb{D}^{1,2} = \left\{ X \in \bigcap_{i=0}^{+\infty} \mathbb{D}^{1,2}_{m_i}, \quad \sum_{i=0}^{+\infty} \lVert D_{m_i} X\rVert_{L^2(\Omega)}^2 < + \infty\right\}$$
and
$$\forall X,Y \in \mathbb{D}^{1,2} \quad \mathcal{E}(X,Y) = \sum_{i=0}^{+\infty} \mathbb{E}[D_{m_i} X D_{m_i} Y].$$
We also note $\mathcal{E}(X) = \mathcal{E}(X,X)$.
\begin{Proposition} \label{propo:local:dirichlet}
    The bilinear form $(\mathbb{D}^{1,2}, \mathcal{E})$ is a local Dirichlet form admitting a   gradient $D$ given by
    \begin{eqnarray}{\label{eqD}}
    \forall X,Y \in \mathbb{D}^{1,2},\ DX = \sum_{i=0}^{+\infty} D_{m_i} X m_i \in L^2(\Omega, \mathcal{H}),
    \end{eqnarray}
   with a carré du champ operator $\Gamma$ given by 
    $$\forall X,Y \in \mathbb{D}^{1,2};\ \Gamma[X,Y] = \langle DX, DY\rangle_\mathcal{H}.$$
    As a consequence $\mathbb{D}^{1,2}$ is a Hilbert space equipped with the norm
    $$\lVert X \rVert^2_{\mathbb{D}^{1,2}} = \lVert X \rVert^2_{L^2(\Omega)} + \mathcal{E}(X).$$
    Moreover, as $\mathcal{S}$ is dense in each $\mathbb{D}_{m_i}^{1,2}, i\in \mathbb{N}$, $\mathcal{S}$ is dense in $\D^{1,2}$.
\end{Proposition}
\begin{proof} 
The fact that $(\mathbb{D}^{1,2}, \mathcal{E})$ is a Dirichlet form with gradient $D$ given by \eqref{eqD} is standard, see for example \cite[Proposition 4.2.1, Chapter 2]{bouleau:hirsch}. \\
Let $\Phi: \R\longrightarrow\R$ be $C^1$ with bounded derivative and $F\in \mathbb{D}^{1,2}$. For for each $i$, the Dirichlet form  $(\mathbb{D}^{1,2}_{m_i},\mathcal{E} )$ is local, hence $D_{m_i}\Phi (F)= \Phi^\prime (F)\cdot D_{m_i}F$. From this we immediately get that $D\Phi (F)=\Phi^\prime (F)\cdot DF$ which proves that $(\mathbb{D}^{1,2}, \mathcal{E})$ is a local Dirichlet form.
\end{proof}

\begin{Corollary} \label{coro:deriveTj}
   For all $j\in\mathbb{N}^*$, evoke that $\overline{T}_j = T_j \wedge T$, we have
    $$ D \overline{T}_j = \dfrac{\overline{T}_j}{T} - \mathbf{1}_{[0, \overline{T}_j]}.$$
    As consequence for all
    $$F = a 1_{\{N_T = 0\}} + \sum_{n=1}^d f_n(T_1, \cdots, T_n) \mathbf{1}_{\{N_T = n\}} \in \mathcal{S},$$
    we have $F \in \mathbb{D}^{1,2}$ and
    $$D F = \sum_{n=1}^d \sum_{j=1}^n \dfrac{\partial f_n}{\partial t_j} (T_1, \cdots, T_n) \left( \dfrac{T_j}{T} - \mathbf{1}_{[0,T_j]} \right) \mathbf{1}_{\{N_T = n\}}.$$
    In particular this expression does not depend on the basis $(m_i)_{i\in \mathbb{N}}$.
\end{Corollary}

\begin{proof}
    We have
    \begin{eqnarray*}
        D \overline{T}_j & = & \sum_{i=0}^{\infty} D_{m_i} \overline{T}_j m_i
   =- \sum_{i=0}^{\infty} \widehat{m}_i(\overline{T}_j) m_i
   = - \sum_{i=0}^{\infty} \left( \int_0^T m_i(s) \mathbf{1}_{\{0\leq s\leq \overline{T}_j\}} ds \right) m_i
        \\ & = & \sum_{i=0}^{\infty} \left\langle \dfrac{\overline{T}_j}{T} -\mathbf{1}_{[0,\overline{T}_j]}, m_i \right\rangle_\mathcal{H} m_i
 = \dfrac{\overline{T}_j}{T} - \mathbf{1}_{[0, \overline{T}_j]}.
    \end{eqnarray*}
    Notice that the term $\dfrac{\overline{T}_j}{T}$ is mandatory to belong to $\mathcal H$ defined by \eqref{eq:def_mathcal_H}. 
\end{proof}
\begin{Corollary}
    For any $F\in \mathbb{D}^{1,2}$ and any $m\in \mathcal{H}$, $D_m F =\langle DF,m\rangle_{\mathcal{H}}$.
\end{Corollary}
\begin{proof}
    Let $j\in\mathbb{N}^*$. The equality is easy to verify for $F=\overline{T_j}$ and then for $F\in \mathcal{S}$. The general case is obtained by a density argument.
\end{proof}
As a consequence of the local property, following \cite[Section 2.2.2]{bouleau:denis:energy:image}    we have similarly to Proposition \ref{functCalculus}:

\begin{Corollary} \label{rem:D:by:parts}
    Let $n\in\N^*$, for any $F=(F_1, \dots, F_n) \in \left(\mathbb{D}^{1,2}\right)^n$ and two functions $\Phi$ and $ \Psi$ defined on $\mathbb{R}^n$, of class $C^1$ with bounded derivatives, 
    $\Phi(F)$ belongs to $\mathbb{D}^{1,2}$
    and
    $$D \Phi(F) = \sum_{j=1}^n \dfrac{\partial \Phi}{\partial x_j}(F) D F_j,$$
    $$\Gamma [\Phi(F),\Psi (F)]=\sum_{1\leq i,j\leq n} \dfrac{\partial \Phi}{\partial x_i} (F)\dfrac{\partial \Psi}{\partial x_j} (F)\Gamma [F_i ,F_j ].$$
\end{Corollary}

\subsection{Divergence operator by duality}

Let $\delta : L^2(\Omega, \mathcal{H}) \longrightarrow L^2(\Omega)$ be the adjoint operator of $D$. Its domain, $\text{Dom}(\delta)$, is the set of $u \in L^2(\Omega, \mathcal{H})$ such that there exists $c\in \mathbb{R}_+^*$ with 
$$\forall F\in \mathbb{D}^{1,2}, \quad \left| \mathbb{E} \left[ \int_0^T D_t F u_t dt \right] \right| \leq c \lVert F \rVert_{\mathbb{D}^{1,2}}.$$
Hence,  for all $u\in \text{Dom}(\delta)$, $\delta (u)$ is the unique element in $L^2(\Omega)$ such that
$$\forall F\in \mathbb{D}^{1,2},  \quad \mathbb{E}[\delta(u) F] = \mathbb{E}[\langle u,DF\rangle_\mathcal{H}] = \mathbb{E} \left[ \int_0^T u_t D_t F dt \right].$$
We now introduce the set $\widetilde{\mathcal{S}}$ of elementary processes $u$ of the form 
$$u =\sum_{i=1}^n A_i m_i, \quad n\in \mathbb{N}^*, \quad A_i \in \mathbb{D}^{1,2}.$$

\begin{Proposition}{\label{FormuleDiv}}
    { For all $u$ in $\widetilde{\mathcal{S}}$, we have $u \in \text{Dom}(\delta)$ and
    $$\delta(u) = \int_{(0,T]} (\psi(u, s) + \widehat{u}(s)  \left(  \Gamma_1(s) + \Gamma_2(s)\right)  +  u(s)) dN_s - \int_0^T D_t (u(t)) dt$$
    where $\psi$, $\Gamma_1$ and $\Gamma_2$ are   defined in Proposition \ref{prop:derivG}.}
\end{Proposition}

\begin{proof}
    Let $u = A m_{i_0}$ with $A \in \mathbb{D}^{1,2}$ and $i_0 \in \mathbb{N}$. For any $F\in \mathbb{D}^{1,2}$ we have 
    \begin{eqnarray*}  
        \left| \mathbb{E} \left[ \int_0^T D_t F u_t dt \right] \right| & \leq & \mathbb{E}\left[ \int_0^T |D_t F| |A m_{i_0}(t)| dt \right]
        \\ & \leq & \left\| F \right\|_{\mathbb{D}^{1,2}} \sqrt{\mathbb{E} \left[ \int_0^T |A|^2 |m_{i_0}(t)|^2dt \right]}
        \\ & = & \left\| F \right\|_{\mathbb{D}^{1,2}} \left\| A \right\|_{L^2(\Omega)} \left\|m_{i_0}\right\|_{L^2(0,T)}
        = c \left\| F \right\|_{\mathbb{D}^{1,2}} 
    \end{eqnarray*}
    with $c = \left\| A \right\|_{L^2(\Omega)} \left\|m_{i_0}\right\|_{L^2(0,T)}$. Thus $u \in \text{Dom}(\delta)$ and for any $F \in \mathbb{D}^{1,2}$
    \begin{eqnarray*}
        \mathbb{E}[\delta(u) F] & = & \mathbb{E} \left[ \int_0^T u_t D_t F dt \right]
   = \mathbb{E}\left[ A \int_0^T m_{i_0}(t) D_t F dt \right]
        \\ & = & \mathbb{E} \left[  A \int_0^T m_{i_0}(t) \sum_{i=0}^{+\infty} D_{m_i} F m_i(t) dt \right]
        \\ & = & \mathbb{E} \left[ A \sum_{i=0}^{+\infty} D_{m_i}F \int_0^T m_{i_0}(t) m_i(t) dt \right]
        \\ & = & \mathbb{E} \left[ A \sum_{i=0}^{+\infty} D_{m_i} F \langle m_{i_0}, m_i \rangle_\mathcal{H} \right]
        = \mathbb{E} [A D_{m_{i_0}} F].
    \end{eqnarray*}
    Thus, integrating by parts: 
    \begin{eqnarray*}
        \mathbb{E}[\delta(u) F] & = & \mathbb{E} [D_{m_{i_0}} (AF)] - \mathbb{E}[F D_{m_{i_0}} A]
        \\ & = & \mathbb{E}\left[ \dfrac{\partial G_{m_{i_0}}^\varepsilon}{\partial \varepsilon}_{|\varepsilon = 0} AF \right] - \mathbb{E}[F D_{m_{i_0}} A]
        \\ & = & \mathbb{E}\left[ \left( \dfrac{\partial G^\varepsilon_{m_{i_0}}}{\partial \varepsilon}_{|\varepsilon = 0} A - D_{m_{i_0}} A \right) F \right].
    \end{eqnarray*}
    Therefore, because $\langle DA, m_{i_0}\rangle_\mathcal{H} = \sum_{i=0}^{+\infty} D_{m_i}A \langle m_i, m_{i_0}\rangle_\mathcal{H} = D_{m_{i_0}} A$,
    \begin{eqnarray*}
        \delta(u) & = & \dfrac{\partial G^\varepsilon_{m_{i_0}}}{\partial \varepsilon}_{|\varepsilon = 0} A - D_{m_{i_0}} A
        \\ & = & { \left(\int_{(0,T]} \left(\psi(m_{i_0}, s)  + \widehat{m_{i_0}}(s) \left(  \Gamma_1(s) + \Gamma_2(s) \right) + m_{i_0}(s)\right)) dN_s\right) A - \int_0^T m_{i_0}(t) D_t A dt}
        \\ & = & {  \int_{(0,T]} (\psi(Am_{i_0}, s) + \widehat{A m_{i_0}}(s) \left(  \Gamma_1(s) + \Gamma_2(s) \right) +  Am_{i_0}(s)) dN_s - \int_0^T D_t (A m_{i_0}(t)) dt}
        \\ & = & {  \int_{(0,T]} (\psi(u, s) + \widehat{u}(s) \left(  \Gamma_1(s) + \Gamma_2(s) \right) +  u(s)) dN_s - \int_0^T D_t (u(t)) dt.}
    \end{eqnarray*}
    We deduce the result for any $u \in \widetilde{\mathcal{S}}$ by linearity.
\end{proof}

\begin{Remark} \label{remark:resume}
    We can retain that:
    \begin{enumerate}
        \item For all $m\in \mathcal{H}$ and $A \in \mathbb{D}^{1,2},$
        $\delta(mA) = \delta(m) A - D_m A.$
        \item For all $m\in \mathcal{H}$ and $A,F \in \mathbb{D}^{1,2}$,
        $\mathbb{E}[A D_m F] = \mathbb{E}[F \delta(mA)].$
        \item { For all $m \in \mathcal{H}$ we have $D_t m_t = 0$
        and
         $$\delta(m) = \int_{(0,T]} (\psi(m, s) + \widehat{m}(s)  \left( \Gamma_1 (s)+\Gamma_2 (s)\right) +  m(s)) dN_s.$$}
    \end{enumerate}
\end{Remark}

\begin{Remark}
    Contrary to the standard Malliavin calculus on the Wiener space (see \cite{nual:06}), we do not a priori have the inclusion of $\mathbb{D}^{1,2} \otimes\mathcal{H}$ in $\text{Dom} (\delta)$ (see Example 3.4 in \cite{Carlen1990} where $\mu = 0$).
\end{Remark}


{  \begin{Corollary} \label{coro:predictable}
    If $u \in L^2(\Omega, \mathcal{H})$ is a predictable process then 
    $$\delta(u) = \int_{(0,T]} (\psi(u, s) + \widehat{u}(s) \left( \Gamma_1 (s)+\Gamma_2 (s)\right) +  u(s)) dN_s$$
\end{Corollary}}

\begin{proof}
    We establish this result for an elementary process of the form:
    $$u(t) = f_01_{[0, t_{1}]}(t)+\sum_{j=1}^{n-1} f_j(\overline{T}_1, \cdots, \overline{T}_{n}) \mathbf{1}_{(t_{j}, t_{j+1}]}(t)$$
     where $n\in \mathbb{N}^*$, $t_j=\frac{jT}{n}$, $f_0$ is a constant,  for any $j\in \{ 1,\cdots,n\}$, $f_j$ is an infinitely differentiable function from $\R^n$ into $\R$ vanishing outside the simplex 
     $$\Delta_n^j =\{ (x_1 ,\cdots ,x_n)\in\R^n,\ 0<x_1<\cdots <x_n\leq t_{j}\}$$
     and  $f_0+\sum_{j=1}^{n-1}f_j(\overline{T}_1, \cdots, \overline{T}_{n}) =0$. This last condition ensures that $u$ belongs to $L^2 (\Omega;\mathcal{H})$. 
      We have, by Proposition \ref{FormuleDiv} and Corollary \ref{coro:deriveTj},
    \begin{eqnarray*}
        D_t u(t) & = & \sum_{j=1}^{n-1} \sum_{i=1}^n \dfrac{\partial f_j}{\partial x_i}(\overline{T}_1, \cdots, \overline{T}_{n}) D_t \overline{T}_i \mathbf{1}_{(t_{j}, t_{j+1}]}(t)
        \\ & = & \sum_{j=1}^{n-1}\sum_{i=1}^{n} \dfrac{\partial f_j}{\partial x_i}(\overline{T}_1, \cdots, \overline{T}_{n}) \left( \dfrac{\overline{T}_i}{T} - \mathbf{1}_{[0, \overline{T}_i]}(t) \right) \mathbf{1}_{(t_{j}, t_{j+1}]}(t).
    \end{eqnarray*}
Let $t\in [0,T]$ and $j$ such that $t\in (t_j ,t_{j+1}]$. Since $f_j$ vanishes outside $\Delta_n^j$, we have for any $j\in\{1,\cdots ,n-1\}$ and any $i\in\{1,\cdots ,n\}$
   $$\dfrac{\partial f_j}{\partial x_i}(\overline{T}_1, \cdots, \overline{T}_{n}) \mathbf{1}_{[0, \overline{T}_i]}(t) \mathbf{1}_{(t_{j}, t_{j+1}]}(t) = 0,$$
   hence
   $$D_t u(t)=\sum_{i=1}^{n} \dfrac{\partial f_j}{\partial x_i}(\overline{T}_1, \cdots, \overline{T}_{n})  \dfrac{\overline{T}_i}{T},$$
   but since $f_0+\sum_{j=1}^{n-1}f_j =0$ we have $\sum_{i=1}^{n} \dfrac{\partial f_j}{\partial x_i}=0$ wich proves that $D_tu_t =0$ in that case.

    Hence from Proposition \ref{FormuleDiv}  we deduce that 
$$\delta (u)=\int_{(0,T]} (\psi(u, s) + \widehat{u}(s) \left( \Gamma_1 (s)+\Gamma_2 (s)\right) +  u(s)) dN_s .$$
    We conclude by using a density argument. Indeed if $u\in L^2(\Omega, \mathcal{H})$ is a predictable process, there exists a sequence $(u_n)_{n\in \mathbb{N}}$ of elementary processes as above converging to $u$ in $L^2(\Omega, \mathcal{H})$.  Clearly the sequence of random variables $\left( \int_{(0,T]} (\psi(u_n, s) + \widehat{u_n}(s)  \Gamma_1 (s) +  u_n(s)) dN_s\right)_n$ converges in $L^2(\Omega )$ to $\int_{(0,T]} (\psi(u, s) + \widehat{u_n}(s)  \Gamma_1 (s) +  u(s)) dN_s$. Now since $\Gamma_2$ is bounded, there exists a constant $C>0$ such that we  have for all $n\in\N$
    $$\left|\int_{(0,T]}\widehat{u_n}(s)  \Gamma_2 (s)dN_S-\int_{(0,T]}\widehat{u}(s)  \Gamma_2 (s)dN_S\right| \leq C
    \int_{(0,T]}\left|\widehat{u_n}(s)-\widehat{u}(s)\right|dN_s.$$
    Hence $\left( \int_{(0,T]} \widehat{u_n}(s)  \Gamma_2 (s)dN_s \right)_n $ converges in $L^2 (\Omega )$ to $\int_{(0,T]}\widehat{u}(s)  \Gamma_2 (s)dN_s$. This proves that 
    $\left( \delta (u_n)\right)_n$ converges to $\int_{(0,T]} (\psi(u, s) + \widehat{u}(s) \left( \Gamma_1 (s)+\Gamma_2 (s)\right) +  u(s)) dN_s$. Since $\delta$ is a closed operator, we conclude that $u$ belongs to $\text{Dom}(\delta )$ and that 
    $$\delta (u)=\int_{(0,T]} (\psi(u, s) + \widehat{u}(s)\left( \Gamma_1 (s)+\Gamma_2 (s)\right) +  u(s)) dN_s .$$
\end{proof}

\begin{Proposition}
    Let $F \in \mathbb{D}^{1,2}$ and $X \in \text{Dom}(\delta)$ such that 
    $$F\delta(X) - \int_0^T D_t F X_t dt \in L^2(\Omega).$$
    Then $FX \in \text{Dom}(\delta)$ and
    $$\delta(FX) = F\delta(X) - \int_0^T D_t FX_t dt.$$
\end{Proposition}

\begin{proof}
    For any $G \in \mathcal{S}$
    \begin{align*}
        & \mathbb{E}[\delta(FX) G] =  \mathbb{E}\left[ \int_0^T F X_t D_t G dt \right]
         =  \mathbb{E}\left[ \int_0^T X_t (D_t (GF) - G D_t F) dt \right]
        \\ & \quad =  \mathbb{E}\left[ \delta(X) GF - G \int_0^T X_t D_t F dt \right] =  \mathbb{E}\left[ G \left( F\delta(X) - \int_0^T D_t F X_t dt \right) \right].
    \end{align*}
\end{proof}

In particular if $X = m \in \mathcal{H} \subset \text{Dom}(\delta)$ then { $\delta(m) = \int_{(0,T]} (\psi(m, s) + \widehat{m}(s) \left( \Gamma_1 (s)+\Gamma_2 (s)\right) +  m(s)) dN_s$ and $\int_0^T D_t F m(t) dt = D_m F$. Hence we have
$$\delta(mF) = F \int_{(0,T]} (\psi(m, s) + \widehat{m}(s) \left( \Gamma_1 (s)+\Gamma_2 (s)\right) +  m(s)) dN_s - D_m F.$$}

\begin{Remark} \label{DN=0}
    We do not have the Clark-Ocone formula because for $F = N_T \in \mathbb{D}^{1,2}$ we have $N_T \neq \mathbb{E}[N_T]$ and $D_t N_T = 0$. Indeed, for any $m\in \mathcal{H}$ and $\varepsilon \in \mathbb{R}_+^*, \mathcal{T}_\varepsilon N_T = N_T$. 
\end{Remark}

\section{Absolute continuity criterion} \label{section:criterion}

\subsection{Local criterion}

Now fix $n\in \mathbb{N}^*$, as usual  $C^{\infty}(\R^n )$ denotes the set of infinitely differentiable functions on $\R^n$. Evoke the density $k_n$ defined by \eqref{eq:cond_dens_jump_times} in Lemma \ref{lemma:density}. We consider the following quadratic form on $C^{\infty}(\R^n )$:
$$e_n(u,v) = \dfrac{1}{2} \int_{\mathbb{R}^n} \sum_{i,j = 1}^n \dfrac{\partial u}{\partial t_i}(t) \dfrac{\partial v}{\partial t_j}(t) \left( t_i \wedge t_j - \dfrac{t_i t_j}{T} \right) k_n(t) dt$$
and
\begin{eqnarray*}
    e_n(u) & = & e_n(u,u).
\end{eqnarray*}

\begin{Proposition} \label{propo:critere:local}~
    \begin{enumerate}
        \item $(C^{\infty}(\R^n ), e_n)$ is closable, its closure $(d_n, e_n)$ defines a local Dirichlet form on $L^2 (k_n(t) dt)$ and each $u\in d_n$ is a $\mathcal{B}(\mathbb{R}^n)$-measurable function in $L^2(k_n(t) dt)$ such that for any $i\in \{1, \cdots, n\}$ and for almost all $$\widetilde{t} = (t_1, \cdots, t_{i-1}, t_{i+1}, \cdots, t_n) \in \mathbb{R}^{n-1},$$ the function $$s\longmapsto u^{(i)}_{\widetilde{t}}(s) = u(t_1, \cdots, t_{i-1},s, t_{i+1}, \cdots, t_n)$$ has an absolutely continuous version $\widetilde{u}_{\widetilde{t}}^{(i)}$ on $[t_{i-1}, t_{i+1}]$ such that
        $$\sum_{i,j = 1}^n \dfrac{\partial u}{\partial t_i}(t) \dfrac{\partial u}{\partial t_j}(t) \left(t_i \wedge t_j - \dfrac{t_i t_j}{T} \right) \in L^1(k_n(t) dt)$$
        where $\dfrac{\partial u}{\partial t_i} = \dfrac{\partial \widetilde{u}_{\widetilde{t}}^{(i)}}{\partial s}$. 
        \item The Dirichlet form $(d_n, e_n)$  admits a carré du champ operator $\gamma_n$ and a gradient operator $\widetilde{D}^n$ given by
        $$\gamma_n[u,v](t) = \sum_{i,j = 1}^n \dfrac{\partial u}{\partial t_i}(t) \dfrac{\partial v}{\partial t_j}(t) \left( t_i \wedge t_j - \dfrac{t_i t_j}{T} \right)$$
        and
        $$\widetilde{D}_s^nu(t) = \sum_{i=1}^n \dfrac{\partial u}{\partial t_i}(t) \left( \dfrac{t_i}{T} - \mathbf{1}_{[0,t_i]}(s) \right)$$
        for all $u,v \in d_n, t = (t_1, \cdots, t_n) \in \mathbb{R}^n$ and $s \in [0,T]$.
        \item The structure $(\mathbb{R}^n, \mathcal{B}(\mathbb{R}^n), k_n(t) dt, d_n, \gamma_n)$ satisfies for every $d \in \mathbb{N}^*, u = (u_1, \cdots, u_d) \in (d_n)^d$,
        $$u_*[\det(\gamma_n[u]) \cdot k_n \nu_n] \ll \nu_d$$
        where $\gamma_n[u]$ denotes the matrix $(\gamma_n(u_i, u_j))_{1\leq i,j\leq d}$, $\nu_n$ (resp. $\nu_d$) the Lebesgue measure on $\mathbb{R}^n$ (resp. $\mathbb{R}^d$) and $u_*[\det(\gamma_n[u]) \cdot k_n \nu_n]$ the image measure defined by, for any $B \in \mathcal{B}(\mathbb{R}^d)$,
        \begin{eqnarray*}
            (u_*[\det(\gamma_n[u]) \cdot k_n \nu_n])(B) & = & [\det(\gamma_n[u]) \cdot k_n \nu_n](u^{-1}(B))
            \\ & = & \int_{u^{-1}(B)} \det(\gamma_n[u,u](t)) k_n(t) dt.
        \end{eqnarray*}
    \end{enumerate}
\end{Proposition}

\begin{proof}
    We prove this result thanks to \cite[Proposition 2.30 and Theorem 2.31]{bouleau:denis} with
    $$k = k_n, \quad d = \tilde{d}_n, \quad \xi_{ij}(t) = t_i \wedge t_j - \dfrac{t_i t_j}{T},$$
    where $\tilde{d}_n$ is the set of $\mathcal{B}(\mathbb{R}^n)$-measurable functions $u\in L^2(k_n(t) dt)$ such that for any $i\in \{1, \cdots, n\}$ and for almost all $$\widetilde{t} = (t_1, \cdots, t_{i-1}, t_{i+1}, \cdots, t_n) \in \mathbb{R}^{n-1},$$ the function $$s\longmapsto u^{(i)}_{\widetilde{t}}(s) = u(t_1, \cdots, t_{i-1},s, t_{i+1}, \cdots, t_n)$$ has an absolutely continuous version $\widetilde{u}_{\widetilde{t}}^{(i)}$ on $[t_{i-1}, t_{i+1}]$ such that
$$\sum_{i,j = 1}^n \dfrac{\partial u}{\partial t_i}(t) \dfrac{\partial u}{\partial t_j}(t) \left(t_i \wedge t_j - \dfrac{t_i t_j}{T} \right) \in L^1(k_n(t) dt)$$
where $\dfrac{\partial u}{\partial t_i} = \dfrac{\partial \widetilde{u}_{\widetilde{t}}^{(i)}}{\partial s}$ and set for any $u,v\in \tilde{d}_n$:
$$e_n(u,v) = \dfrac{1}{2} \int_{\mathbb{R}^n} \sum_{i,j = 1}^n \dfrac{\partial u}{\partial t_i}(t) \dfrac{\partial v}{\partial t_j}(t) \left( t_i \wedge t_j - \dfrac{t_i t_j}{T} \right) k_n(t) dt.$$

    The function $k = k_n : \mathbb{R}^n \longrightarrow \mathbb{R}_+$ is measurable and the functions $\xi_{i,j}$ are symmetric Borel function. We have to check if the two conditions  (HG) of \cite{bouleau:denis} are satisfied.
    
\noindent  { {\bf Condition 1.} Let  $i\in \{1, ..., n\}$ and $\tilde t = (t_1, \cdots, t_{i-1}, t_{i+1}, \cdots, t_n)$ with $0<t_1 < \cdots < t_{i-1} < t_{i+1} < \cdots < t_n < T$. We clearly have $ k_{n,\overline{t}}^{(i)}(s)=0$ if $s\notin ]t_{i-1},t_{i+1}[$. Moreover,  since for all $s\in ]t_{i-1},t_{i+1}[$, 
$k_{n,\overline{t}}^{(i)}(s)\geq \lambda_* >0$, $\left( k_{n,\overline{t}}^{(i)}\right)^{-1}$ is locally integrable on $]t_{i-1},t_{i+1}[$. So, condition 1 is satisfied.}

\noindent {\bf Condition 2.}  For any $t = (t_1, \cdots, t_n) \in \mathbb{R}^n$ and any $c\in \mathbb{R}^n$, such that $0=t_0 < t_1 < t_2 < \cdots < t_n < T=t_{n+1}$,
        \begin{eqnarray*}
            \sum_{i,j=1}^n \xi_{ij}(t) c_i c_j & = & \sum_{i,j = 1}^n \left(t_i \wedge t_j - \dfrac{t_i t_j}{T} \right) c_i c_j
            \\ & = & \dfrac{1}{T}\sum_{i = 1}^n t_i \left(T - t_i \right) c_i^2+  \dfrac{2}{T} \sum_{1 \leq i < j\leq n} t_i \left(T - t_j \right) c_i c_j
        \end{eqnarray*}
        This double sum can be split as follows:
        \begin{align*}
            &\sum_{i = 1}^n t_i \left(T - t_i \right) c_i^2 = \sum_{i = 1}^n  \sum_{k=1}^i (t_k-t_{k-1}) \sum_{\ell = i}^n (t_{\ell+1} - t_\ell ) c_i^2 \\
            & \quad =  \sum_{k=1}^n (t_k-t_{k-1}) \sum_{\ell = k}^n (t_{\ell+1} - t_\ell ) \sum_{i=k}^\ell c_i^2 \\
            & \quad = \sum_{k=1}^n (t_k-t_{k-1})(t_{k+1}-t_k) c_k^2 + \sum_{k=1}^{n-1} (t_k-t_{k-1}) \sum_{\ell = k+1}^n (t_{\ell+1} - t_\ell ) \sum_{i=k}^\ell c_i^2
        \end{align*}
        and 
        \begin{align*}
            & 2 \sum_{1 \leq i < j\leq n} t_i \left(T - t_j \right) c_i c_j = 2 \sum_{1 \leq i < j\leq n}  \sum_{k=1}^i (t_k-t_{k-1}) \sum_{\ell = j}^n (t_{\ell+1} - t_\ell )c_ic_j \\
            & \quad = 2 \sum_{k=1}^{n-1} (t_k-t_{k-1}) \sum_{k \leq i < j\leq n}  \sum_{\ell = j}^n (t_{\ell+1} - t_\ell )c_ic_j \\
            & \quad = 2 \sum_{k=1}^{n-1} (t_k-t_{k-1})\sum_{\ell = k+1}^n (t_{\ell+1} - t_\ell ) \sum_{k \leq i < j\leq \ell}  c_ic_j 
        \end{align*}
        Coming back to the initial sum, we have 
        \begin{eqnarray*}
            \sum_{i,j=1}^n \xi_{ij}(t) c_i c_j & = & \dfrac{1}{T} \sum_{k=1}^n (t_k-t_{k-1})(t_{k+1}-t_k) c_k^2 \\
            &  + & \dfrac{1}{T} \sum_{k=1}^{n-1} (t_k-t_{k-1}) \sum_{\ell = k+1}^n (t_{\ell+1} - t_\ell ) \left[  \sum_{i=k}^\ell c_i^2 + 2 \sum_{k \leq i < j\leq \ell}  c_ic_j \right] \\
            & = & \dfrac{1}{T} \sum_{k=1}^n (t_k-t_{k-1})(t_{k+1}-t_k) c_k^2 \\
            & +&  \dfrac{1}{T} \sum_{k=1}^{n-1} (t_k-t_{k-1}) \sum_{\ell = k+1}^n (t_{\ell+1} - t_\ell ) \left[  \sum_{i,j=k}^\ell (c_i+c_j)^2  \right] \\
            & \geq & \dfrac{1}{T} \sum_{k=1}^n (t_k-t_{k-1})(t_{k+1}-t_k) c_k^2.
        \end{eqnarray*}
        Thus, for any compact $K \subset \{(t_1, \cdots, t_n), \quad 0 = t_0 < t_1 < \cdots < t_n < T = t_{n+1}\}$, there exists $c \in \mathbb{R}_+^*$ such that, for any $(t_1, \cdots, t_n) \in K$,
        $$\sum_{i,j=1}^n \xi_{ij}(t) c_i c_j \geq \dfrac{c^2}{T} \sum_{k=1}^n c_k^2.$$

The hypotheses $(HG)$ of \cite[Proposition 2.30]{bouleau:denis} being satisfied, we conclude that $(\tilde{d}_n ,e_n)$ is a local Dirichlet hence $(C^{\infty} (\R^n ), e_n )$ is closable and its closure $(d_n ,e_n)$ is such that $d_n\subset \tilde{d}_n $. The last assertion is a consequence of \cite[Theorem 2.31]{bouleau:denis}.
\end{proof}

Using this result, we consider $\lVert \cdot \rVert_{d_n}$ the norm on $d_n$ defined by: 
    $$\forall u \in d_n,\ \lVert u \rVert^2_{d_n} = \lVert u \rVert^2_{L^2(k_n(t)dt)} + 2 e_n(u).$$

\subsection{Global criterion}

    We remind that for any $F \in L^0(\Omega)$, there exists $a \in \mathbb{R}$ and $f_n : \mathbb{R}^n \longrightarrow \mathbb{R}, n\in \mathbb{N}^*,$ measurable such that, $\mathbb{P}$-almost surely,
    \begin{equation} \label{eq:measurable}
        F = a 1_{\{N_T = 0\}} + \sum_{n=1}^{\infty} f_n(T_1, \cdots, T_n) \mathbf{1}_{\{N_T = n\}}.
    \end{equation}

\begin{Proposition} \label{propo:decomposition}
    Let $F \in L^0(\Omega)$ of the form \eqref{eq:measurable}. Then $F \in \mathbb{D}^{1,2}$ if and only if $f_n \in d_n$ for any $n\in \mathbb{N}^*$ and
    $$\sum_{n=1}^{
    \infty} \lVert f_n \rVert^2_{d_n} \mathbb{P}(N_T = n) < \infty.$$
    In this case
    $$\lVert F \rVert_{\mathbb{D}^{1,2}}^2 = a^2 \mathbb{P}(N_T = 0) + \sum_{n=1}^{\infty} \lVert f_n \rVert^2_{d_n} \mathbb{P}(N_T = n).$$
\end{Proposition}

\begin{proof}~
    Let $F = a 1_{\{N_T = 0\}} + \sum_{n = 1}^d  f_{n}(T_1, T_2, \cdots, T_n)  \mathbf{1}_{\{N_T = n\}}$ be in $\mathcal{S}$. Then as a consequence of Corollary \ref{coro:deriveTj}, $F$ belongs to $\D^{1,2}$ and 
$$
        \mathcal{E}(F,F)=\E \left[\int_0^T |D_s F|^2\, ds\right]=2\sum_{n = 1}^d  \mathbb{P}(N_T =n)e_n (f_n).
$$
    Hence
    $$\lVert F \rVert_{\mathbb{D}^{1,2}}^2 = a^2 \mathbb{P}(N_T = 0) + \sum_{n=1}^{d} \lVert f_n \rVert^2_{d_n} \mathbb{P}(N_T = n).$$
    We conclude using a density argument. Indeed, if $F$ belongs to $\D^{1,2}$, there exists a sequence $(F^k)_k$ in $\mathcal{S}$ converging to $F$ in $\D^{1,2}$. Now if for any $k$
    $$F^k = a_k 1_{\{N_T = 0\}} + \sum_{n = 1}^{\infty}  f^k_{n}(T_1, T_2, \cdots, T_n)  \mathbf{1}_{\{N_T = n\}}$$
    with $f_n^k \in C^{\infty}(\R^d )$ and $f_n^k=0$ for $n$ large, then clearly for all $n,k,\ell$:
    $$\| f^k_n -f^\ell_n\|_{d_n}^2\P( N_T =n)\leq \| F^k -F^\ell\|_{\D^{1,2}}^2$$
    hence $(f^k_n )_n$ converges to an element $f_n$ in $d_n$, $a^k$ tends to a real number $a$ and we get that 
    $$F=a 1_{\{N_T = 0\}} + \sum_{n = 1}^{\infty}  f_{n}(T_1, T_2, \cdots, T_n) \mathbf{1}_{\{N_T = n\}},$$
    and 
    $$\lVert F \rVert_{\mathbb{D}^{1,2}}^2 =\lim_{m\rightarrow +\infty}\lVert F 1_{\{N_T \leq m \}}\rVert_{\mathbb{D}^{1,2}}^2= a^2 \mathbb{P}(N_T = 0) + \sum_{n=1}^{\infty} \lVert f_n \rVert^2_{d_n} \mathbb{P}(N_T = n).$$
    Conversely, if  $F \in L^0(\Omega)$ of the form \eqref{eq:measurable} is such that  $f_n \in d_n$ for any $n\in \mathbb{N}^*$ and
    $$\sum_{n=1}^{\infty} \lVert f_n \rVert^2_{d_n} \mathbb{P}(N_T = n) < +\infty,$$
    then define for any $m\in\mathbb{N}^*$, $F^m=a 1_{\{N_T = 0\}} + \sum_{n = 1}^{m}  f_{n}(T_1, T_2, \cdots, T_n)  \mathbf{1}_{\{N_T = n\}}$ by approximating each $f_n$ for $n\in \{1,\cdots ,m\}$ by a sequence of functions in $C^{\infty} (\R^d)$ we easily get that $F^m$ belong to $\D^{1,2}$ and 
    $$\lVert F^m \rVert_{\mathbb{D}^{1,2}}^2 = a^2 \mathbb{P}(N_T = 0) + \sum_{n=1}^{m} \lVert f_n \rVert^2_{d_n} \mathbb{P}(N_T = n).$$
     Then $(F^m)$ is a Cauchy sequence in $\D^{1,2}$ converging to $F$ in $L^2$, this ends the proof.
\end{proof}

\begin{Remark}
    We can summarize the one-to-one correspondence between the Dirichlet structure $(\Omega, \mathcal{F}_T, \mathbb{P}, \mathbb{D}^{1,2}, \Gamma)$ and the finite dimensional structures 
    $$(\mathbb{R}^n, \mathcal{B}(\mathbb{R}^n), k_n(t) dt, d_n, \gamma_n),$$ $n\in \mathbb{N}^*$: for any $F \in \mathbb{D}^{1,2}$ of the form \eqref{eq:measurable} and a.e. $s\in [0,T]$:
    \begin{enumerate}
        \item  $\lVert F \rVert_{L^2(\Omega)}^2 = a^2 \mathbb{P}(N_T = 0) + \sum_{n=1}^{\infty} \lVert f_n\rVert^2_{L^2(k_n(t)dt)},$
        \item  $D_s F = \sum_{n=1}^{\infty} \widetilde{D}_s^n f_n(T_1, \cdots, T_n)  \mathbf{1}_{\{N_T = n\}},$
        with
        $$\widetilde{D}_s^n f_n(T_1, \cdots, T_n) = \sum_{i=1}^n \dfrac{\partial f_n}{\partial t_i}(T_1, \cdots, T_n) \left( \dfrac{T_i}{T} -  \mathbf{1}_{[0,T_i]}(s) \right),$$
        \item 
        $\Gamma[F] = \sum_{n=1}^{\infty} \gamma_n[f_n](T_1, \cdots, T_n)  \mathbf{1}_{\{N_T = n\}},$
        \item
        $\mathcal{E}(F) = \sum_{n=1}^{\infty} e_n(f_n) \mathbb{P}(N_T = n),$
        \item
        $\lVert F\rVert^2_{\mathbb{D}^{1,2}} = a^2 \mathbb{P}(N_T = 0) + \sum_{n=1}^{\infty} \lVert f_n \rVert^2_{d_n} \mathbb{P}(N_T = n).$
    \end{enumerate}    
\end{Remark}

\begin{Theorem} \label{thm:abs_cont}
    Let $d\in \mathbb{N}^*$ and $F = (F_1, \cdots, F_d) \in (\mathbb{D}^{1,2})^d$. Then, noting
    $$\Gamma[F] = (\Gamma[F_i,F_j])_{1\leq i,j\leq d},$$
    the image measure $F_* [\det(\Gamma[F]).\mathbb{P}]$ is absolutely continuous with respect to the Lebesgue measure $\nu_d$ on $\mathbb{R}^d$.
\end{Theorem}

\begin{proof}
    Let $B \subset \mathbb{R}^d$ such that $\nu_d(B) = 0$. We would like to get
    \begin{eqnarray*}
        0 & = & (F_*[\det(\Gamma[F]).\mathbb{P}])(B)
       = \int_{F^{-1}(B)} \det(\Gamma[F](\omega)) d\mathbb{P}(\omega)
        \\ & = & \int_\Omega \det(\Gamma[F](\omega)) 1_B(F(\omega)) d\mathbb{P}(\omega)
       = \mathbb{E} \left[ \det(\Gamma[F]) 1_B(F) \right].
    \end{eqnarray*}
But, according to Proposition \ref{propo:decomposition}, there exist $a\in \mathbb{R}$ and $f_n \in (d_n)^d, n\in \mathbb{N}^*$ such that
    $$F = a 1_{\{N_T = 0\}} + \sum_{n=1}^{\infty} f_n(T_1, \cdots, T_n)  \mathbf{1}_{\{N_T = n\}}.$$
    Thus
$$
        \Gamma[F] = \Gamma[F,F]
        \\ = \sum_{n=1}^{\infty} \gamma_n[f_n](T_1, \cdots, T_n)  \mathbf{1}_{\{N_T = n\}}.
$$
    As a consequence 
    $$\det(\Gamma[F])  \mathbf{1}_{\{N_T = 0\}} = 0,$$
    and for any $n \in \mathbb{N}^*$,
    \begin{eqnarray*}
        \det(\Gamma[F]) \mathbf{1}_{\{N_T = n\}} & = & \det(\gamma_n[f_n](T_1, \cdots, T_n))  \mathbf{1}_{\{N_T = n\}}.
    \end{eqnarray*}
    Therefore
    \begin{eqnarray*}
        \det(\Gamma[F]) 1_B(F) & = & \sum_{n=1}^{\infty} \det(\gamma_n[f_n](T_1, \cdots, T_n))  \mathbf{1}_B(f_n(T_1, \cdots, T_n))  \mathbf{1}_{\{N_T = n\}}
    \end{eqnarray*}
    and, according to Lemma \ref{lemma:density},
    \begin{eqnarray*}
        && \mathbb{E}\left[ \det(\Gamma[F]) 1_B(F) \right]
        \\ && = \sum_{n=1}^{\infty} \mathbb{E}\left[ \det(\gamma_n[f_n](T_1, \cdots, T_n)  \mathbf{1}_B(f_n(T_1, \cdots, T_n)  \mathbf{1}_{\{N_T = n\}} \right]
        \\ && = \sum_{n=1}^{\infty} \mathbb{E}\left[ \det(\gamma_n[f_n](T_1, \cdots, T_n)  \mathbf{1}_B(f_n(T_1, \cdots, T_n) \mid N_T = n \right] \mathbb{P}(N_T = n)
        \\ && = \sum_{n=1}^{\infty} \left( \int_{\mathbb{R}^n} \det(\gamma_n[f_n](t))  \mathbf{1}_B(f_n(t)) k_n(t) dt \right) \mathbb{P}(N_T = n)
        \\ && = \sum_{n=1}^{\infty} ((f_n)_*[\det(\gamma_n[f_n] \cdot k_n \nu_n)])(B) \mathbb{P}(N_T = n).
    \end{eqnarray*}
    However, according to Proposition \ref{propo:critere:local} applied to $f_n \in (d_n)^d$, the measure $(f_n)_*[\det(\gamma_n[f_n] \cdot k_n \nu_n)]$ is absolutely continuous with respect to $\nu_d$. Thus, for any $n\in \mathbb{N}^*$,
    $$((f_n)_*[\det(\gamma_n[f_n] \cdot k_n \nu_n)])(B) = 0$$
    and
    $$\mathbb{E}\left[ \det(\Gamma[F])  \mathbf{1}_B(F) \right] = 0.$$
    This concludes the proof.
\end{proof}

\begin{Corollary} \label{coro:criter:density}
    Let $d \in \mathbb{N}^*$ and $F = (F_1, \cdots, F_d) \in (\mathbb{D}^{1,2})^d$. Then, conditionally to the fact that $\Gamma[F]$ is invertible, the law of the random variable $F$ is absolutely continuous with respect to the Lebesgue measure $\nu_d$.
\end{Corollary}

\section{Applications}

\subsection{SDE and density of the solution} \label{sect:sde:density}

We consider the stochastic differential equation

\begin{equation} \label{sde}
    X_t = x_0 + \int_0^t f(s,X_s) ds + \int_{(0,t]} g(s,X_{s-}) dN_s, \quad 0\leq t\leq T,
\end{equation}
or in the differential form
$$dX_t = f(t,X_t) dt + g(t,X_{t-}) dN_t, \quad X_0 = x_0.$$
We assume that:
\begin{Assumption} \label{assumption:sde}
    The functions $f,g : [0,T] \times \mathbb{R}^d \longrightarrow \mathbb{R}^d$ are measurable and satisfy
    \begin{enumerate}
        \item For any $t\in [0,T]$, the maps $f(t,\cdot), g(t,\cdot)$ are of class $C^1$.
        \item $\sup_{t,x} (|\nabla_x f(t,x)| + |\nabla_x g(t,x)|) < +\infty$.
        \item For any $x\in \mathbb{R}^d$, the map $g(\cdot,x)$ is differentiable.
    \end{enumerate}
\end{Assumption}

\begin{Remark}
    Since $N_T$ admits moments of any order (See \cite{leblanc}), according to \cite[Chapter V.3 Theorem 7 and Chapter V.4 Theorem 10]{prot:04}, there exists a unique solution $X$ to \eqref{sde} such that $\sup_{0\leq t\leq T} |X_t| \in L^2(\Omega)$.
\end{Remark}

We consider the deterministic flow $\Phi$ defined by the solution of the ordinary differential equation
$$\Phi_{s,t}(x) = x + \int_s^t f(u,\Phi_{s,u}(x))du, \quad 0\leq s \leq t\leq T, \quad x\in \mathbb{R}^d.$$

\begin{Proposition} \label{propo:sol:explicit}
    On the set $\{N_T = 0\}$, we have
    $$X_t = \Phi_{0,t}(x_0), \quad 0\leq t\leq T.$$
    And, for any $n\in \mathbb{N}^*$, on the set $\{N_T = n\}$, we have
    $$X_t = [\Phi_{T_n,t} \circ \Psi(T_n, \cdot) \circ \cdots \circ \Phi_{T_1,T_2} \circ \Psi(T_1, \cdot) \circ \Phi_{0,T_1}](x_0), \quad T_n \leq t \leq T,$$
    where 
    $$\Psi(t,x) = x + g(t,x), \quad 0\leq t\leq T, \quad x\in \mathbb{R}^d.$$
\end{Proposition}
\begin{proof}
On the set $\{N_T = 0\}$ we have 
        $$X_t = x_0 + \int_0^t f(s,X_s) ds, \quad 0\leq t\leq T.$$
        Thus for all $t \in [0,T]$, 
        $X_t = \Phi_{0,t}(x_0)$. On the $\{N_T = 1\}$ we have, for any $t \in [0,T_1)$,
        $$X_t = x_0 + \int_0^t f(s,X_s) ds.$$
        Thus
        $X_t = \Phi_{0,t}(x_0)$.
        Then for any $t\in [T_1,T]$ we have 
        \begin{eqnarray*}
            X_t & = & X_{T_1-} + \int_{T_1}^t f(s,X_s) ds + \int_{[T_1,t]} g(s,X_{s-}) dN_s
            \\ & = & X_{T_1-} + \int_{T_1}^t f(s,X_s) ds + g(T_1, X_{T_1-}) 
            \\ & = & \Psi(T_1, X_{T_1-}) + \int_{T_1}^t f(s,X_s) ds.
        \end{eqnarray*}
        Hence
        \begin{eqnarray*}
            X_t & = & \Phi_{T_1,t}(\Psi(T_1, X_{T_1-}))
          = \Phi_{T_1,t}(\Psi(T_1,\Phi_{0,T_1}(x_0)))
            \\ & = & [\Phi_{T_1,t} \circ \Psi(T_1,\cdot) \circ \Phi_{0,T_1} ](x_0).
        \end{eqnarray*}
        Then we proceed by induction on $\{N_T= n\}$. 
\end{proof}

We continue this section with the same ideas as in \cite{denis:nguyen}. As the following results are formal computations, they are proved in the same way.

\begin{Remark}
    The process $\nabla_x \Phi$ satisfies, for any $0\leq s\leq t\leq T$ and $x\in \mathbb{R}^d$,
    $$\dfrac{\partial}{\partial t} \nabla_x \Phi_{s,t}(x) = \nabla_x f(t,\Phi_{s,t}(x)) \nabla_x \Phi_{s,t}(x), \quad \nabla_x \Phi_{s,s}(x) = I_d.$$
    Thus
    $$\nabla_x \Phi_{s,t}(x) = \exp\left( \int_s^t \nabla_x f(u, \Phi_{s,u}(x)) du \right), \quad 0\leq s\leq t\leq T.$$
\end{Remark}

\begin{Definition} \label{def:K}
    We define the process $K$ as the derivative of the flow generated by $X$, solution of the SDE
    $$K_t = I_d + \int_0^t \nabla_x f(s,X_s) K_s ds + \int_{(0,t]} \nabla_x g(s,X_{s-}) K_{s-} dN_s, \quad 0\leq t\leq T.$$
\end{Definition}

From now we assume that:

\begin{Assumption}
    For any $(t,x)\in [0,T]\times \R^d$, 
    $$\det(I_d + \nabla_x g(t,x)) \neq 0$$
    and $(I_d + \nabla_x g)^{-1}$ is uniformly bounded.
\end{Assumption}

We now define the process $\widetilde{K}$ as the solution of 
\begin{eqnarray*}
    \widetilde{K}_t & = & I_d - \int_0^t \widetilde{K}_s \nabla_x f(s,X_s) ds
    \\ && - \int_{(0,t]} \widetilde{K}_s \nabla_x g(s,X_{s-}) (I_d- \nabla_x g(s,X_{s-})(I_d + \nabla_x g(s,X_{s-}))^{-1}) dN_s .
\end{eqnarray*}
Following \cite[Proposition 8.7]{bouleau:denis}, we have:

\begin{Lemma} \label{lemma:Kts}
    Processes $K$ and $\widetilde{K}$ satisfy
    $$K_t \widetilde{K}_t = I_d, \quad 0\leq t\leq T.$$
    Moreover:
    $$K_{T_i} =(I_d + \nabla_x g(t,X_{T_i-}))K_{T_i-}, \quad \widetilde{K}_{T_i} = (I_d + \nabla_x g(T_i, X_{T_i-}))^{-1} \widetilde{K}_{T_i-}, \quad i\in \mathbb{N}^*.$$
\end{Lemma}

\begin{Definition}
    We define the process $(K_t^s)_{0\leq s\leq t\leq T}$ by:
    $$K_t^s = K_t \widetilde{K}_s, \quad 0\leq s\leq t\leq T.$$
\end{Definition}
\noindent Similarly to \cite[Proposition 6.4]{denis:nguyen}, we get the following result.

\begin{Proposition} \label{propo:derive:sde}
    Let $\varphi : [0,T] \times \mathbb{R}^d \longrightarrow \mathbb{R}^d$ defined by:
    $$\forall (t,x) \in [0,T] \times \mathbb{R}^d, \quad \varphi(t,x) = f(t,x+g(t,x)) - (I_d + \nabla_x g(t,x)) f(t,x) - \dfrac{\partial g}{\partial t}(t,x).$$
    Then $X_T \in \mathbb{D}^{1,2}$ and, for a.e. $s\in [0,T]$,
    $$D_s X_T = - \int_{(0,T]} K_T^t \varphi(t,X_{t-}) \left(\dfrac{t}{T} - 1_{[0,t]}(s) \right) dN_t.$$
    Moreover
    $$\Gamma[X_T] = \int_{(0,T]} \int_{(0,T]} K_T^t \varphi(t,X_{t-}) (\varphi(u,X_{u-}))^* (K_T^u)^* \left( u\wedge t - \dfrac{ut}{T}  \right) dN_t dN_u.$$
\end{Proposition}
This yields:

\begin{Corollary} \label{coro:densiy:XT}
    If we consider the event
    $$\mathcal{C} = \left\{ \det\left( \int_{(0,T]} \int_{(0,T]} K_T^t \varphi(t,X_{t-}) (\varphi(u,X_{u-}))^* (K_T^u)^* \left( u\wedge t - \dfrac{ut}{T} \right) dN_t dN_u \right) > 0 \right\}$$
    then if $\mathbb{P}(\mathcal{C}) > 0$, the law of $X_T$ conditionally to $\mathcal{C}$ is absolutely continuous with respect to the Lebesgue measure on $\mathbb{R}^d$.
\end{Corollary}

\begin{proof}
    It is a direct consequence of Proposition \ref{propo:derive:sde} and Corollary \ref{coro:criter:density}.
\end{proof}

\begin{Theorem} \label{thm:density:d=1}
    In dimension $d = 1$, if for any $t \in [0,T]$ and $x\in \mathbb{R}$,
    $$\varphi(t,x) = f(t,x+g(t,x)) - f(t,x) - \dfrac{\partial g}{\partial x}(t,x) f(t,x) - \dfrac{\partial g}{\partial t}(t,x) \neq 0,$$
    then, conditionally to $\{N_T\geq 1\}$, the law of $X_T$ is absolutely continuous with respect to the Lebesgue measure on $\mathbb{R}$.
\end{Theorem}

\begin{proof}
    As $d = 1$ we have, according to Proposition \ref{propo:derive:sde},
    $$D_s X_T =-\sum_{i=1}^{N_T}K_T^{T_i}\varphi (T_i ,X_{T_i -} ) \left(\dfrac{T_i}{T} - \mathbf{1}_{[0,T_i]}(s) \right).$$
    Let $\omega \in \Omega$ such that
    $$\Gamma [X_T] (\omega)=\int_0^T |D_s X_T  (\omega)|^2 ds = 0,$$
    then, for almost every $s\in [0,T],  D_s X_T  (\omega) = 0$. Thus, for almost every $s \in [T_{N_T (\omega)}, T],$ we get, writing $T_i$ instead of $T_i(\omega)$,
    $$\sum_{i=1}^{N_T (\omega)}K_T^{T_i}\varphi (T_i ,X_{T_i -} ) \dfrac{T_i}{T} = 0.$$
    Then for almost every $s \in [T_{N_T (\omega)-1}, T_{N_T (\omega)}]$
    $$\sum_{i=1}^{N_T (\omega)-1}K_T^{T_i}\varphi (T_i ,X_{T_i -} ) \dfrac{T_i}{T} +K_T^{T_{N_T(\omega)}}\varphi (T_{N_T(\omega)} ,X_{T_{N_T(\omega)} -}) \left(\dfrac{T_{N_T(\omega)}}{T} - 1\right) = 0.$$
    Therefore, by subtracting the two equations, we get
    $$\varphi (T_{N_T (\omega)} ,X_{T_{N_T (\omega)} -} ) = 0$$
    then
    $$\sum_{i=1}^{N_T (\omega)-1}K_T^{T_i}\varphi (T_i ,X_{T_i -} ) \dfrac{T_i}{T} = 0.$$
    Thus, by considering $s \in [T_{N_T (\omega)-2}, T_{N_T (\omega)-1}]$ then $s \in [T_{N_T (\omega)-3}, T_{N_T (\omega)-2}]$, ..., we get, by successive iterations, for any $i\in \{1, \cdots, N_T(\omega)\}$,
    $$\varphi(T_i, X_{T_i-}) = 0.$$
    Therefore, by contrapositive, if $\varphi(t,x) \neq  0$ for any $t \in [0,T]$ and $x\in \mathbb{R}$ then $\Gamma [X_T]>0$ on the set $\{ N_T \geq 1\}$. Thus, according to Corollary \ref{coro:criter:density}, conditionally to $\{N_T \geq 1\}$, the law of $X_T$ is absolutely continuous with respect to the Lebesgue measure on $\mathbb{R}$.
\end{proof}

\begin{Example}[Linear with constant coefficients in dimension $d = 1$]
    We consider $X$ the solution of the linear SDE
    \begin{equation} \label{eq:sde:linear}
        dX_t = (a X_t + b) dt + (\alpha X_{t-} + \beta) dN_t, \ 0\leq t\leq T, \quad X_0 = x_0,
    \end{equation}
    where $x_0, a, b, \alpha, \beta \in \mathbb{R}$ satisfy
    $$a\beta - \alpha b \neq 0.$$
    In this case we have, for any $t \in [0,T]$ and $x\in \mathbb{R}$,
    $$\varphi(t,x) = a(x + \alpha x + \beta) + b - ax - b - \alpha (ax+b) = a\beta - \alpha b \neq 0.$$
    Then, according to Theorem \ref{thm:density:d=1}, conditionally to $\{N_T \geq 1\}$, the law of $X_T$ is absolutely continuous with respect to the Lebesgue measure on $\mathbb{R}$.
\end{Example}

\begin{Corollary} \label{coro:wronskian}
    We assume that $d = 1$ and the parameters $f$ and $g$ do not depend on $t \in [0,T]$. We consider the Wronskian of $f$ and $g$:
    $$W(f,g) = g'\times f - f' \times g.$$
    Thus if the function $f$ is of class $C^2$ and
    $$\forall x\in \mathbb{R}, \quad |W(f,g)(x)| > \dfrac{1}{2} \lVert f''\rVert_\infty \lVert g \rVert^2_\infty$$
    then, conditionally to $\{N_T \geq 1\}$, the law of $X_T$ is absolutely continuous with respect to the Lebesgue measure on $\mathbb{R}$.
\end{Corollary}

\begin{proof}
    As the parameters $f$ and $g$ do not depend on $t$, we have, for any $t\in [0,T]$ and $x\in \mathbb{R}$,
    $$\varphi(t,x) = f(x+g(x)) - f(x) - g'(x) f(x).$$
    Let $x\in\R$. Tthen, by Taylor expansion, there exists $y_x \in \mathbb{R}$ such that
    $$\varphi(t,x) = g(x) f'(x) + \dfrac{1}{2} g(y_x)^2 f''(x) - g'(x) f(x) = \dfrac{1}{2} g(y_x)^2 f''(x) - W(f,g)(x).$$
    Thus, if $\varphi(t,x) = 0$ then, according to the assumption,
    $$|W(f,g)(x)| = \dfrac{1}{2} g(y_x)^2 |f''(x)| \leq \dfrac{1}{2} \lVert g\rVert_\infty^2 \lVert f''\rVert_\infty < |W(f,g)(x)|$$
    which is absurd. Therefore $\varphi(t,x) \neq 0$ for any $t\in [0,T]$ and $x\in \mathbb{R}$. We conclude by using Theorem \ref{thm:density:d=1}.
\end{proof}

\begin{Example}
    We consider $X$ the solution of the SDE in dimension $d = 1$
    \begin{equation}
        dX_t = \cos(X_t) dt + \sin(X_{t-}) dN_t, \ 0\leq t\leq T, \quad X_0 = x_0,
    \end{equation}
    where $x_0 \in \mathbb{R}$. In particular
    $$\forall x\in \mathbb{R}, \quad |W(f,g)(x)| = 1 > \dfrac{1}{2} = \dfrac{1}{2} \lVert \cos''\rVert_\infty \lVert \sin\rVert_\infty^2.$$
    Thus, according to Corollary \ref{coro:wronskian}, conditionally to $\{N_T \geq 1\}$, the law of $X_T$ is absolutely continuous with respect to the Lebesgue measure on $\mathbb{R}$.
\end{Example}

\begin{Proposition} \label{propo:spanning}
    If there exists $\ell \in \mathbb{N}$ such that for any $n \geq \ell, 0\leq t_1 < \cdots < t_n \leq T$ and $x_1, \cdots, x_n \in \mathbb{R}^d$, the family $(K_T^{t_i} \varphi(t_i, x_i))_{1\leq i\leq n}$ spans $\mathbb{R}^d$ then, conditionally to $\{N_T \geq \ell\}$, the law of $X_T$ is absolutely continuous with respect to the Lebesgue measure on $\mathbb{R}^d$.
\end{Proposition}

\begin{proof}
    Let $\omega \in \{N_T \geq \ell\}$ such that $\Gamma[X_T](\omega)$ is non invertible. Then, as it is a nonnegative symetric matrix, there exists $u \in \mathbb{R}^d \backslash \{0\}$ such that
    $$u^* \Gamma[X_T] (\omega) u = \int_0^T (u^* D_s X_T(\omega))^2 ds = 0.$$
    Then, according to Proposition \ref{propo:derive:sde},
    $$0 = \int_0^T (u^* D_s X_T(\omega))^2 ds = \int_0^T \left(u^* \sum_{i=1}^{N_T(\omega)} K_T^{T_i} \varphi(T_i, X_{T_i-})\left( \dfrac{T_i}{T} -  \mathbf{1}_{[0,T_i]}(s) \right) \right)^2 ds.$$
    Thus, for almost every $s \in [0,T]$,
    $$u^* \sum_{i=1}^{N_T(\omega)} K_T^{T_i} \varphi(T_i, X_{T_i-})\left( \dfrac{T_i}{T} -  \mathbf{1}_{[0,T_i]}(s) \right) = 0.$$
    We deduce, as in dimension $d = 1$, that, for any $i\in \{1, \dots, N_T(\omega)\}$,
    $$u^* K_T^{T_i} \varphi(T_i, X_{T_i -}) = 0$$
    which is absurd because $(K_T^{T_i} \varphi(T_i, X_{T_i-}))_{1\leq i\leq N_T(\omega)}$ spans $\mathbb{R}^d$.
\end{proof}

\begin{Remark}
    Necessarily we have $\ell \geq d$.
\end{Remark}

\begin{Example}[Linear with constant coefficients in dimension $d \in \mathbb{N}^*$]

    We consider $X$ the solution of the linear SDE
    \begin{equation} 
        dX_t = (A X_t + b) dt + (M X_{t-} + \beta) dN_t, \ 0\leq t\leq T, \quad X_0 = x_0,
    \end{equation}
    where $x_0, b, \beta \in \mathbb{R}^d$ and $A,M \in \mathbb{R}^{d\times d}$.\\
  We assume that 
    $$\det(M + I_d) \neq 0, \quad AM = MA$$
    and that there exists $\ell \in \mathbb{N}$ such that, for any $n \geq \ell$ and $0\leq t_1 < \cdots < t_n \leq T$, the family 
    $$(\exp(A(T-t_i)) (I_d + M)^{n-i}(A\beta - Mb))_{1\leq i\leq n}$$
    spans $\mathbb{R}^d$. Then, conditionally to $\{N_T \geq \ell\}$, the law of $X_T$ is absolutely continuous with respect to the Lebesgue measure on $\mathbb{R}^d$.\\
    Indeed, in that case:
    \begin{itemize}
        \item The process $K$ defined by Definition \ref{def:K} satisfies
        $$K_t = I_d + \int_0^t A K_s ds + \int_{(0,t]} M K_{s-} dN_s, \quad 0\leq t\leq T,$$
        i.e.
        $$dK_t = A K_s ds + M K_{s-} dN_s, \quad 0\leq t\leq T, \quad K_0 = I_d,$$
        i.e., for any $t \in [0,T]$, as $AM = MA$,
        $$K_t = \exp\left( A t \right) \prod_{0< s\leq t} \left(I_d + M \Delta N_s \right) = \exp(At) (I_d+M)^{N_t}.$$
        \item For any $t \in [0,T]$, since matrices $A$ and $M$ commute:
        $$\widetilde{K}_t=(K_t)^{-1}=\exp(-At) (I_d + M)^{-N_t}.$$
        \item The process $(K_t^s)$ is equal to, for any $0\leq s\leq t\leq T$,
        \begin{eqnarray*}
            K_t^s & = & K_t \widetilde{K}_s
            \\ & = & \exp(A(t-s)) (I_d+M)^{N_t - N_s}.
        \end{eqnarray*}
        \item According to Proposition \ref{propo:derive:sde}, we consider the function $\varphi : [0,T] \times \mathbb{R} \longrightarrow\mathbb{R}$ defined by, for any $t\in [0,T]$ and $x\in \mathbb{R}$,
        \begin{eqnarray*}
            \varphi(t,x) & = & A(x+M x + \beta) + b -(I_d+M)(A x + b) 
            \\ & = & A \beta - M b.
        \end{eqnarray*}
    \end{itemize}
    Therefore,  we get the absolute continuity criterion as a consequence of Proposition \ref{propo:spanning}.
    
    As an example of matrices $A$ and $M$ satisfying the assumptions, we can take  $A = I_d$ the identity matrix, $M$  a diagonalizable matrix with $d$ distinct eigenvalues (different to $-1$ to have $\det(M + I_d) \neq 0$): there exists $\lambda_1, \cdots, \lambda_d$ in $\mathbb{R}\backslash\{-1\}$ (distinct) and $P \in GL_d(\mathbb{R})$ such that
    $$M = P \left( \begin{array}{ccc} \lambda_1 & & (0) \\ & \ddots & \\ (0) & & \lambda_d \end{array} \right) P^{-1},$$
    and $v:= [P^{-1}(\beta - Mb)]_j = [P^{-1} \beta - DP^{-1} b]_j \neq 0$ for any $j\in \{1, \cdots, d\}$. Thus, for any $n \geq \ell, 0\leq t_1 < \cdots < t_n \leq T$ and $i\in \{1, \cdots, n\}$,
    \begin{eqnarray*}
        && \exp(A(T-t_i)) (I_d + M)^{n-i}(A\beta - Mb)
        \\ && = e^{T-t_i} P \left( \begin{array}{ccc} (1+\lambda_1)^{n-i} & & (0) \\ & \ddots & \\ (0) & & (1+\lambda_d)^{n-i} \end{array} \right) P^{-1} (\beta - Mb).
    \end{eqnarray*}
    Therefore the family $(\exp(A(T-t_i)) (I_d + M)^{n-i}(A\beta - Mb))_{1\leq i\leq n}$ spans $\mathbb{R}^d$ if and only if the family 
    $$\left( \left( \begin{array}{c} (1+\lambda_1)^{n-1} v_1 \\ \vdots \\ (1+\lambda_d)^{n-1} v_d \end{array} \right), \cdots, \left( \begin{array}{c} (1+\lambda_1) v_1 \\ \vdots \\ (1+\lambda_d) v_d \end{array} \right), \left( \begin{array}{c} v_1 \\ \vdots \\ v_d \end{array} \right) \right)$$
    spans $\mathbb{R}^d$. The determinant of the last $d$ vectors of this family is related to a Vandermonde determinant:
    $$v_1 \cdots v_d \prod_{1\leq i<j\leq d} (\lambda_j - \lambda_i)$$
    which is not null. Therefore, conditionally to $\{N_T \geq \ell\}$, the law of $X_T$ admits an absolutely continuous law with respect to the Lebesgue measure on $\mathbb{R}^d$.
\end{Example}

\subsection{Application to Greek computation}
\subsubsection{The model we consider}
We consider an asset price whose dynamics is given by
\begin{equation} \label{eq:wealth}
    dS_t = r S_t dt + \sigma S_{t-} d\widetilde{N}_t, \quad S_0 = x_0,
\end{equation}
where $\widetilde{N}$ is the compensated Hawkes process equal to $\widetilde{N}_t = N_t - \int_0^t \lambda^*(s) ds$, $r$ the interest rate, $\sigma$ the volatility and $x_0$ the initial wealth. In other words we have
$$dS_t = (r - \sigma \lambda^*(t))S_t dt + \sigma S_{t-} dN_t.$$
Thus, if we write $\alpha_t = r - \sigma \lambda^*(t)$ then the dynamic is equivalent to
\begin{equation} \label{eq:important}
    dS_t = \alpha_t S_t dt + \sigma S_{t-} dN_t, \quad S_0 = x_0.
\end{equation}
We also consider an European option
$$C = \mathbb{E}\left[ f(S_T) \right]$$
with $f$ a function which can be not continuous as $f = 1_{[K,+\infty[}$ for example for a binary European option, and we interest to compute associated Greeks
$$\Delta = \dfrac{\partial C}{\partial x_0}, \quad \Gamma = \dfrac{\partial^2 C}{\partial x^2_0}, \quad \rho = \dfrac{\partial C}{\partial r}, \quad \nu = \dfrac{\partial C}{\partial \sigma}.$$
In the sequel $x$ denotes a real number in an interval $(a,b)$ which can be equal to $x_0, r$ or $\sigma$.
For a family of random variables $F=(F^x)_{a<x<b}$, 
we consider a class $\mathcal{L}(F)$ of real functions $f$ of the form
\begin{equation} \label{eq:not:regular}
    f(y) = \sum_{i=1}^n \Phi_i(y) 1_{A_i}(y), \quad y\in \mathbb{R},
\end{equation}
where $n\in \mathbb{N}$, $\Phi_i$ is continuous and bounded and $A_i$ is an interval with endpoints in $\mathbb{T}= \mathbb{T}_0 \cup \{-\infty, \infty\}$ with the set $\mathbb{T}_0 \subset \mathbb{R}$ defined by:
$$\mathbb{T}_0 = \left\{ y \in \mathbb{R} \quad \lim_{n \to \infty} \sup_{a<x<b} \mathbb{P} \left(F^x \in \left( y-\dfrac{1}{n}, y+\dfrac{1}{n} \right) \right) = 0 \right\}.$$ 

\begin{Proposition} \label{propo:calcul:greeks}
    Let $(F^x)_{a<x<b}$ and $(G^x)_{a<x<b}$ be two families of random variables such that the maps $x \in (a,b) \longmapsto F^x \in \mathbb{D}^{1,2}$ and $x\in (a,b) \longmapsto \mathbb{D}^{1,2}$ are continuously differentiable. Let $m \in \mathcal{H}$ such that for any $x \in (a,b)$ on $\left\{ \dfrac{\partial F^x}{\partial x} \neq 0 \right\}$
    $$D_m F^x \neq 0$$
    and such that $m G^x \dfrac{\frac{\partial F^x}{\partial x}}{D_m F^x}$ is continuous in $x$ in $\text{Dom}(\delta)$. Thus for any $f \in \mathcal{L}(F)$ the map $x \longmapsto \mathbb{E}[f(F^x)]$ is continuous differentiable and
    \begin{equation}\label{Dependancem}
   \dfrac{\partial}{\partial x} \mathbb{E}[G^x f(F^x)] = \mathbb{E}\left[ f(S_T^{x_0}) \delta \left(G^x m \dfrac{\frac{\partial F^x}{\partial x}}{D_m F^x} \right) \right] + \mathbb{E}\left[ \dfrac{\partial G^x}{\partial x} f(F^x) \right]. \end{equation}
\end{Proposition}

\begin{proof}
    We follow the proof of \cite[Proposition 7.2]{denis:nguyen}. First if $f \in C_b^1(\mathbb{R})$ then, as the maps $x\longmapsto F^x$ and $x\longmapsto G^x$ are differentiable, the map $x\longmapsto G^x f(F^x)$ is differentiable and, for any $x\in (a,b)$,
    \begin{eqnarray*}
        \dfrac{\partial}{\partial x}(G^x f(F^x)) & = & \dfrac{\partial G^x}{\partial x} f(F^x) + \dfrac{\partial F^x}{\partial x} f'(F^x) G^x.
    \end{eqnarray*}
    Then, as $f \in C^1_b(\mathbb{R})$, $D_m f(F^x) = f'(F_x) D_m F^x$, and, as $D_m F^x \neq 0$ on $\left\{ \dfrac{\partial F^x}{\partial x} \neq 0 \right\}$,
    \begin{eqnarray*}
        \dfrac{\partial}{\partial x}(G^x f(F^x)) & = & \dfrac{\partial G^x}{\partial x} f(F^x) + \dfrac{\partial F^x}{\partial x} \dfrac{D_m f(F^x)}{D_m F^x} G^x.
    \end{eqnarray*}
    Thus, according to Remark \ref{remark:resume} and $m G^x \dfrac{\frac{\partial F^x}{\partial x}}{D_m F^x} \in \text{Dom}(\delta)$,
    \begin{eqnarray*}
        \mathbb{E}\left[\dfrac{\partial}{\partial x}(G^x f(F^x))\right] & = & \mathbb{E}\left[\dfrac{\partial G^x}{\partial x} f(F^x)\right] + \mathbb{E}\left[ G^x \frac{\dfrac{\partial F^x}{\partial x}}{D_m F^x} D_m f(F^x) \right]
        \\ & = & \mathbb{E}\left[\dfrac{\partial G^x}{\partial x} f(F^x)\right] + \mathbb{E}\left[ \delta\left(m G^x \frac{\dfrac{\partial F^x}{\partial x}}{D_m F^x}\right) f(F^x) \right].
    \end{eqnarray*}
    Finally, as the function $x \longmapsto \dfrac{\partial}{\partial x} (G^x f(F^x))$ is continuous, we have
    \begin{eqnarray*}
        \dfrac{\partial}{\partial x}\mathbb{E}[G^x f(F^x)] & = & \mathbb{E}\left[\dfrac{\partial}{\partial x}(G^x f(F^x))\right]
    \end{eqnarray*}
    and the result follows in that case. For the general case $(f\in \mathcal{L}(F))$ we conclude as in \cite[Proposition 7.2]{denis:nguyen}, using an approximation argument.
\end{proof}

\begin{Remark}
One can remark that in relation \eqref{Dependancem}, the right hand side term depends on the choice of $m$. For numerical purpose this choice should be investigated and is left for further researches.
\end{Remark}
As in \cite{khatib:privault}, if we can apply the Proposition \ref{propo:calcul:greeks} with
$$F^x = S_T^x, \quad G^x =  \mathbf{1}_{\{N_T > 0\}}, \quad x \in (a,b),$$
then we can compute the Delta conditionally to $\{N_T\geq 1\}$
$$\dfrac{\partial}{\partial x_0} \mathbb{E}\left[  \mathbf{1}_{\{N_T > 0\}} f\left(S_T^{x_0}\right) \right] = \mathbb{E}\left[ f(S_T^{x_0}) \delta\left(m  \mathbf{1}_{\{N_T > 0\}} \dfrac{\frac{\partial S_T^{x_0}}{\partial x_0}}{D_m S_T^{x_0}} \right) \right].$$
The solution of \eqref{eq:wealth}
$$dS_t = S_t \alpha_t dt + \sigma S_{t-} dN_t, \quad S_0 = x_0,$$
is given by
$$S_t^{x_0} = x_0 \exp\left( \int_0^t \alpha_s ds \right) (1+\sigma)^{N_t} = x_0 \exp \left( rt - \sigma \int_0^t \lambda^*(s) ds \right)(1+\sigma)^{N_t}, \quad 0\leq t\leq T.$$
In particular
$$S_T^{x_0} = x_0 \exp\left( rT - \sigma \Lambda_T \right) (1+\sigma)^{N_T}$$
with  {  
\begin{eqnarray*}
    \Lambda_T & = & \int_0^T \lambda^*(s) ds
= \int_0^T\left( \lambda_s+ \gamma \left( \sum_{i=1}^{N_T}  \mu(s-T_i)  \mathbf{1}_{\{T_i \leq s\}}\right)\right) ds
\end{eqnarray*}
}
Thus the map $x_0 \longmapsto S_T^{x_0}$ is continuously differentiable and
$$\dfrac{\partial S_T^{x_0}}{\partial x_0} = \exp \left( rT - \sigma \Lambda_T \right)  (1+\sigma)^{N_T} = \dfrac{S_T^{x_0}}{x_0}.$$
To compute $D S_T^{x_0}$ we cannot use the results of Subsection \ref{sect:sde:density} because the paramater of the SDE $f : (t,x) \longmapsto \alpha_t x$ is not deterministic. {    However we first remark that 
$$\Lambda_T =\int_0^T \lambda_s ds  \mathbf{1}_{\{ \N_T =0\}}+\sum_{n=1}^{+\infty}\int_0^T \lambda^* (s; T_1 ,\cdots ,T_n )ds  \mathbf{1}_{\{ N_T =n\}},$$
and following calculations of Section \ref{Section2.5}, we know that for all $n\geq 1$ and $i_0\in \{1,\cdots ,n\}$:
\begin{align*}  \dfrac{\partial}{\partial t_{i_0}}\int_0^T \lambda^*(s;t_1, \dots, t_n) ds & = \gamma \left(\sum_{j=1}^{i_0-1} \mu(t_{i_0}-t_j)\right) - \gamma \left(\sum_{j=1}^{i_0-1} \mu(t_{i_0}-t_j) + \mu(0) \right) \\
    & - \int_{t_{i_0}}^T  \gamma' \left(\sum_{j=1}^n \mu(s-t_j) 1_{\{s > t_j\}}\right) \mu'(s-t_{i_0}) ds.
    \end{align*}}
Thus $\Lambda_T \in \mathbb{D}^0_m$ and
{  $$D_m \Lambda_T =\sum_{n=1}^n D_m \Lambda_T  \mathbf{1}_{\{ \N_T =n\}},$$
with for each $n\geq 1$:
$$D_m \Lambda_T 1_{\{ \N_T =n\}}=
\sum_{i_0=1}^n \dfrac{\partial}{\partial t_{i_0}}\int_0^T \lambda^*(s;T_1, \dots, T_n) ds  \cdot \widehat{m} (T_{i_0})$$
where 
\begin{align*}
& \dfrac{\partial}{\partial t_{i_0}}\int_0^T \lambda^*(s;t_1, \dots, t_n) ds = \gamma \left(\sum_{j=1}^{i_0-1} \mu(T_{i_0}-T_j)\right) \\
&\qquad - \gamma \left(\sum_{j=1}^{i_0-1} \mu(T_{i_0}-T_j) + \mu(0) \right)  - \int_{T_{i_0}}^T  \gamma' \left(\sum_{j=1}^n \mu(s-T_j) 1_{\{s > T_j\}}\right) \mu'(s-T_{i_0}). 
\end{align*}

\subsubsection{The linear case}

In order to simplify the calculations, we now assume that we are in the linear case i.e. $\gamma (x)=x$. 
In that case, we have}
$$D_m \Lambda_T = \sum_{i=1}^{N_T} \mu(T-T_i) \widehat{m}(T_i) = \int_{(0,T]} \mu(T-t) \widehat{m}(t) dN_t.$$
Since $D_m N_T = 0$ (see Remark \ref{DN=0}), we get, by chain rule, $S_T^{x_0} \in \mathbb{D}^{1,2}_m$ and
$$D_m S_T^{x_0} = - \sigma S_T^{x_0} D_m \Lambda_T = -\sigma S_T^{x_0} \sum_{i=1}^{N_T} \mu(T-T_i) \widehat{m}(T_i) = - \sigma S_T^{x_0} \int_{(0,T]} \mu(T-t) \widehat{m}(t) dN_t.$$
Thus we have to choose a function $m \in \mathcal{H}$ such that, for any $t\in [0,T], \widehat{m}(t) = 0$ if and only if $t\in \{0,T\}$. 
In this case we get
$$\mathbf{1}_{\{N_T > 0\}} \dfrac{\frac{\partial S_T^{x_0}}{\partial x_0}}{D_m S_T^{x_0}} = -\dfrac{\mathbf{1}_{\{N_T> 0\}}}{\sigma x_0 \sum_{i=1}^{N_T} \mu(T-T_i) \widehat{m}(T_i)} = - \dfrac{\mathbf{1}_{\{N_T> 0\}}}{\sigma x_0 \int_{(0,T]} \mu(T-t) \widehat{m}(t) dN_t}.$$
Finally, according to Remark \ref{remark:resume}, $m \mathbf{1}_{\{N_T>0\}} \dfrac{\frac{\partial S_T^{x_0}}{\partial x_0}}{D_m S_T^{x_0}} \in \text{Dom}(\delta)$ because we have $m \in \mathcal{H}$ and $\mathbf{1}_{\{N_T>0\}} \dfrac{\frac{\partial S_T^{x_0}}{\partial x_0}}{D_m S_T^{x_0}} \in \mathbb{D}^{1,2}_m$, and
\begin{align*}
    & \delta \left(m \mathbf{1}_{\{N_T> 0\}} \dfrac{\frac{\partial S_T^{x_0}}{\partial x_0}}{D_m S_T^{x_0}} \right)
    \\ & = - \delta(m) \dfrac{\mathbf{1}_{\{N_T> 0\}}}{\sigma x_0 \sum_{i=1}^{N_T} \mu(T-T_i) \widehat{m}(T_i)} + D_m \left( \dfrac{1_{\{N_T> 0\}}}{\sigma x_0 \sum_{i=1}^{N_T} \mu(T-T_i) \widehat{m}(T_i)} \right)
    \\ & = - \delta(m) \dfrac{\mathbf{1}_{\{N_T> 0\}}}{\sigma x_0 \int_{(0,T]} \mu(T-t) \widehat{m}(t) dN_t} + D_m \left( \dfrac{\mathbf{1}_{\{N_T> 0\}}}{\sigma x_0 \int_{(0,T]} \mu(T-t) \widehat{m}(t) dN_t} \right)
\end{align*}
with, on $\{N_T>0\}$,
\begin{align*}
    & D_m \left( \dfrac{\mathbf{1}_{\{N_T> 0\}}}{\sigma x_0 \sum_{i=1}^{N_T} \mu(T-T_i) \widehat{m}(T_i)} \right)
    \\ & = - \dfrac{\sigma x_0 \sum_{i=1}^{N_T} D_m(\mu(T-T_i) \widehat{m}(T_i))}{\sigma^2 x_0^2 \left( \sum_{i=1}^{N_T} \mu(T-T_i) \widehat{m}(T_i)\right)^2} 
    \\ & = - \dfrac{\sum_{i=1}^{N_T} [\widehat{m}(T_i) D_m(\mu(T-T_i)) + \mu(T-T_i) D_m \widehat{m}(T_i)]}{\sigma x_0 \left( \sum_{i=1}^{N_T} \mu(T-T_i) \widehat{m}(T_i)\right)^2}
    \\ & = \dfrac{\sum_{i=1}^{N_T} \widehat{m}(T_i) \mu'(T-T_i) D_mT_i}{\sigma x_0 \left( \sum_{i=1}^{N_T} \mu(T-T_i) \widehat{m}(T_i)\right)^2} - \dfrac{\sum_{i=1}^{N_T} \mu(T-T_i) m(T_i)D_mT_i}{\sigma x_0 \left( \sum_{i=1}^{N_T} \mu(T-T_i) \widehat{m}(T_i)\right)^2}
    \\ & = -\dfrac{\sum_{i=1}^{N_T} \mu'(T-T_i) \widehat{m}(T_i)^2}{\sigma x_0 \left( \sum_{i=1}^{N_T} \mu(T-T_i) \widehat{m}(T_i)\right)^2} + \dfrac{\sum_{i=1}^{N_T} \mu(T-T_i) m(T_i)\widehat{m}(T_i)}{\sigma x_0 \left( \sum_{i=1}^{N_T} \mu(T-T_i) \widehat{m}(T_i)\right)^2}
    \\ & = -\dfrac{\int_{(0,T]} \mu'(T-s) \widehat{m}(s)^2 dN_s}{\sigma x_0 \left( \int_{(0,T]} \mu(T-s) \widehat{m}(s) dN_s\right)^2}+ \dfrac{\int_{(0,T]} \mu(T-s) m(s)\widehat{m}(s) dN_s}{\sigma x_0 \left( \int_{(0,T]} \mu(T-s) \widehat{m}(s) dN_s\right)^2},
\end{align*}
and, using Remark \ref{remark:resume}, 
\begin{align*}
    \delta(m) & =  \int_{(0,T]} (\psi(m, s) + \widehat{m}(s) \mu(T-s) +  m(s)) dN_s
    \\ & =  \sum_{j=1}^{N_T} (\psi(m, T_j) + \widehat{m}(T-T_j) + m(T_j))
    \\ & =  \sum_{j=1}^{N_T} \left( \dfrac{1}{\lambda^*(T_j)}\int_{(0,T_j)} (\widehat{m}(T_j) - \widehat{m}(t) )\mu'(T_j-t) dN_t + \widehat{m}(T-T_j) + m(T_j) \right)
    \\ & =  \sum_{j=1}^{N_T} \left(\dfrac{\sum_{i=1}^{j-1} (\widehat{m}(T_j) - \widehat{m}(T_i))\mu'(T_j - T_i)}{\lambda + \sum_{i=1}^{j-1} \mu(T_j-T_i)} + \widehat{m}(T-T_j) + m(T_j) \right).
\end{align*}
Thus we deduce the following expression of the Delta:
\begin{Proposition} \label{propo:delta}
    We have
    \begin{eqnarray*}
        && \dfrac{\partial}{\partial x_0} \mathbb{E}[\mathbf{1}_{\{N_T\geq 1\}} f(S_T)]
        \\ && = \mathbb{E}\left[ f(S_T^{x_0}) \delta\left(m \mathbf{1}_{\{N_T> 0\}} \dfrac{\frac{\partial S_T^{x_0}}{\partial x_0}}{D_m S_T^{x_0}} \right) \right]
        \\ && = - \mathbb{E}\left[ \dfrac{f(S_T^{x_0}) \delta(m)\mathbf{1}_{\{N_T> 0\}}}{\sigma x_0 \int_{(0,T]} \mu(T-t) \widehat{m}(t) dN_t} \right] - \mathbb{E}\left[ \dfrac{f(S_T^{x_0}) \int_{(0,T]} \mu'(T-s) \widehat{m}(s)^2 dN_s}{\sigma x_0 \left( \int_{(0,T]} \mu(T-s) \widehat{m}(s) dN_s \right)^2}\mathbf{1}_{\{N_T> 0\}}\right]
        \\ && \quad + \mathbb{E}\left[ \dfrac{f(S_T^{x_0}) \int_{(0,T]} \mu(T-s) m(s) \widehat{m}(s) dN_s}{\sigma x_0 \left( \int_{(0,T]} \mu(T-s) \widehat{m}(s) dN_s \right)^2}\mathbf{1}_{\{N_T> 0\}} \right].
    \end{eqnarray*}
\end{Proposition}

\begin{Remark}
    Any term in the expression of $\Delta$ can be written from the Hawkes process $N$, the jump instants $(T_i)_{i\in \mathbb{N}^*}$ and the parameters $T, \lambda, \mu, \widehat{\mu}, \mu', m, \widehat{m}$ and $f$:
    \begin{eqnarray*}
        S_T^{x_0} & = & x_0 \exp\left( rT - \sigma \Lambda_T \right) (1+\sigma)^{N_T},
        \\ \Lambda_T & = & { \int_0^T \lambda_s ds}+ \sum_{i=1}^{N_T} \widehat{\mu}(T-T_i),
        \\ \delta(m) & = & \sum_{j=1}^{N_T} \left(\dfrac{\sum_{i=1}^{j-1} (\widehat{m}(T_j) - \widehat{m}(T_i))\mu'(T_j - T_i)}{{ \lambda_{T_j}} + \sum_{i=1}^{j-1} \mu(T_j-T_i)} + \widehat{m}(T-T_j) + m(T_j) \right),
        \\ \int_{(0,T]} \mu(T-t) \widehat{m}(t) dN_t & = & \sum_{i=1}^{N_T} \mu(T-T_i) \widehat{m}(T_i)
        \\ \int_{(0,T]} \mu'(T-t) \widehat{m}(t)^2 dN_t & = & \sum_{i=1}^{N_T} \mu'(T-T_i) \widehat{m}(T_i)^2
        \\ \int_{(0,T]} \mu(T-t) m(t) \widehat{m}(t) dN_t & = & \sum_{i=1}^{N_T} \mu(T-T_i) m(T_i) \widehat{m}(T_i).
    \end{eqnarray*}
    In other words, if we simulate a sample of Hawkes process then we can approach $\Delta$ conditionally to $\{N_T > 0\}$.
\end{Remark}

\begin{Remark}
    On $\{N_T = 0\}$, the process $S^{x_0}$ is deterministic and we have to know the derivative of the function $f$ to compute $\Delta$.
\end{Remark}

\begin{Remark}
    For the other Greeks we can notice that
    \begin{eqnarray*}
        \dfrac{\partial^2 S_T^{x_0}}{\partial x_0^2} & = & 0,
        \\ \dfrac{\partial S_T^r}{\partial r} & = & T S_T^{x_0}
        \\ \dfrac{\partial S_T^\sigma}{\partial \sigma} & = & \left(- \Lambda_T + \dfrac{N_T}{1+\sigma}\right) S_T^{x_0}.
    \end{eqnarray*}
    Then we can deduce similar expressions of the other Greeks conditionally to $\{N_T > 0\}$. For $\Gamma = \dfrac{\partial^2 C}{\partial x_0^2}$ we can start by writing
    \begin{eqnarray*}
        G^{x_0} & = & m \mathbf{1}_{\{N_T > 0\}} \dfrac{\frac{\partial S_T^{x_0}}{\partial x_0}}{D_m S_T^{x_0}},
        \\ \dfrac{\partial^2}{\partial x_0^2} \mathbb{E}[\mathbf{1}_{\{N_T>0\}} f(S_T^{x_0})] & = & \dfrac{\partial}{\partial x_0} \mathbb{E}\left[f(S_T^{x_0}) \delta\left( G^{x_0} \right)\right]
        \\ & = & \mathbb{E}\left[ f(S_T^{x_0}) \delta\left(G^x m \dfrac{\frac{\partial S_T^{x_0}}{\partial x_0}}{D_m S_T^{x_0}} \right) \right] + \mathbb{E}\left[ f(S_T^{x_0}) \dfrac{\partial}{\partial x_0} \delta\left( G^{x_0} \right) \right]
    \end{eqnarray*}
    where we apply two times Proposition \ref{propo:calcul:greeks} with different processes $G$.
\end{Remark}

\bibliography{biblio}

\end{document}